%% file: observers_paper.tex
\documentclass[a4paper,oneside,11pt,reqno]{amsart}
\usepackage[backend=bibtex,style=ieee,citestyle=numeric-comp,giveninits=true,url=false,eprint=false,doi=true,isbn=false]{biblatex}
\renewbibmacro*{doi+eprint+url}{%
	\printfield{doi}%
	\newunit\newblock%
	\iftoggle{bbx:eprint}{%
		\usebibmacro{eprint}%
	}{}%
	\newunit\newblock%
	\iffieldundef{doi}{%
		\usebibmacro{url+urldate}}%
	{}%
}

\usepackage{hyperref}
\usepackage[foot]{amsaddr}
\usepackage{amsmath}
\numberwithin{equation}{section}
\usepackage{amsthm}
\usepackage[english]{babel}
\usepackage{xurl}
\usepackage{graphicx}
\usepackage{xfrac}

\usepackage{physics}

\renewcommand{\ip}[2]{{\langle #1 , #2\rangle}}

\usepackage{subcaption}
\usepackage{todonotes}
\usepackage{algorithm}
\usepackage{algpseudocode}

\usepackage{empheq}
\usepackage{enumitem}

\usepackage{xcolor}

\usepackage{unicode-math}

\calclayout

\newcommand{\xDom}{\Omega}
\newcommand{\tDom}{[0,\infty)}
\DeclareMathOperator{\hamil}{\symcal{H}}
\DeclareMathOperator{\contG}{\symcal{G}}
\DeclareMathOperator{\disG}{\symcal{G}_\mathrm{h}}
\newcommand{\ppot}{P}
\newcommand{\su}{{u}} 
\newcommand{\hu}{\widehat{u}}
\newcommand{\dens}{\rho}
\newcommand{\flow}{m}
\newcommand{\vel}{v}
\newcommand{\hdens}{\widehat{\rho}}
\newcommand{\hflow}{\widehat{m}}
\newcommand{\hvel}{\widehat{v}}
\newcommand{\xStep}{\mathrm{h}}
\newcommand{\tStep}{\tau}
\newcommand{\numu}{{u}_\mathrm{h}} 
\newcommand{\numhu}{\widehat{u}_\mathrm{h}}
\newcommand{\numdens}{\rho_\mathrm{h}}
\newcommand{\numflow}{m_\mathrm{h}}
\newcommand{\numvel}{v_\mathrm{h}}
\newcommand{\numhdens}{\widehat{\rho}_\mathrm{h}}
\newcommand{\numhflow}{\widehat{m}_\mathrm{h}}
\newcommand{\numhvel}{\widehat{v}_\mathrm{h}}
\newcommand{\numhvelNL}[1]{\sfrac{\numhflow^{#1}}{\numhdens^{#1}}}
\newcommand{\Enthalp}{{\symscr{h}}}
\newcommand{\hEnthalp}{\widehat{\symscr{h}}}
\newcommand{\numEnthalp}{\symscr{h}_\mathrm{h}}
\newcommand{\numhEnthalp}{\widehat{\symscr{h}}_\mathrm{h}}
\newcommand{\nM}{{M}}
\newcommand{\hM}{\widehat{M}}
\newcommand{\numnM}{{M}_\mathrm{h}}
\newcommand{\numhM}{\widehat{M}_\mathrm{h}}
\newcommand{\res}{\mathrm{res}}
\newcommand{\diss}{\symcal{D}}
\newcommand{\besov}[2]{L^{#1}(0,\infty;L^{#2}(\Omega))}
\newcommand{\src}{\hat s}
\newcommand{\error}{\hat\eta}
\newcommand{\numerror}{\hat\eta_\mathrm{h}}

\newcommand{\ddt}{\overline{\partial_\tau}}
\newcommand{\cpoin}{C_\mathrm{poin}}

\theoremstyle{definition}
\newtheorem{definition}{Definition}
\newtheorem{remark}[definition]{Remark}
\theoremstyle{Theorem}
\newtheorem{theorem}[definition]{Theorem}
\newtheorem{lemma}[definition]{Lemma}
\newtheorem{corollary}[definition]{Corollary}

\begin{document}
	\title[Convergence of a fully discrete observer]{Convergence analysis of a fully discrete observer for data assimilation
		of the barotropic Euler equations}
	\author{Aidan Chaumet}
	\email[Aidan Chaumet]{chaumet@mathematik.tu-darmstadt.de}
	\author{Jan Giesselmann}
	\email[Jan Giesselmann]{giesselmann@mathematik.tu-darmstadt.de}
	
	\begin{abstract}
	We study the convergence of a discrete Luenberger observer for the barotropic Euler equations in one dimension, for measurements of the velocity only. We use a mixed finite element method in space and implicit Euler integration in time. We use a modified relative energy technique to show an error bound comparing the discrete observer to the original system's solution. The bound is the sum of three parts: an exponentially decaying part, proportional to the difference in initial value, a part proportional to the grid sizes in space and time and a part that is  proportional to the size of the measurement errors as well as the nudging parameter. The proportionality constants of the second and third parts are independent of time and grid sizes. To the best of our knowledge, this provides the first error estimate for a discrete observer for a quasilinear hyperbolic system, and implies uniform-in-time accuracy of the discrete observer for long-time simulations.

	\end{abstract}
	\date{\today}
	
	\maketitle
	\input{./Sections/introduction.tex}

	\input{./Sections/summary_cont.tex}
	\input{./Sections/conv_proof.tex}
	\input{./Sections/numerics.tex}

	\input{./Sections/conclusion.tex}

	 \section*{Acknowledgements}	
	 The authors are grateful for financial support by the German Science Foundation (DFG) via grant TRR 154 (\textit{Mathematical modelling, simulation and optimization using the example of gas networks}), sub-project C05 (Project 239904186). 

\printbibliography
\end{document}

%% file: Sections/introduction.tex
\section{Introduction}
The barotropic Euler equations  are a system of quasilinear hyperbolic balance laws that model the evolution of compressible, inviscid fluids, for example in aerodynamics and gas transport in pipes. Often, only part of the state can be measured. For example, in aerodynamics,  flow velocities can be measured relatively easily using \textit{particle image velocimetry} (PIV)\cite{raffel2018particle}.  However, measuring densities is not straightforward, especially without altering the flow dynamics.
Then, \textit{data assimilation} combines measurement data with the underlying physical evolution equations, to reconstruct the full system state, including density, from the velocity data alone.

We study a \textit{distributed Luenberger observer} \cite{luenberg_intro} approach for data assimilation, also commonly known as  \textit{nudging} \cite{hoke1977dynamic}. This approach introduces a forcing term, scaled by an arbitrary but fixed nudging parameter, that drives the system towards the measurements.  Motivated by PIV, we assume that the complete velocity field is measured.

In this scenario, on the continuous level, \cite{kunkel_obs} proves exponential \textit{synchronization} for the barotropic Euler equations in one space dimension, that is, convergence of the solution to the observer system towards the original system state in  time, provided that solutions are sufficiently smooth and satisfy suitable smallness assumptions for the state and its time derivatives. However, in practice, there are sources of errors, such as discretization errors from solving the observer system numerically and measurement errors, whose influence must be understood. We choose a discretization that allows us to prove that by combining simulations with data, the observer is uniformly accurate for long times.

We discretize the Euler system in space using a mixed finite element method (MFEM), which is based on rewriting the equations as a port-Hamiltonian system, and use implicit Euler integration in time \cite{egger_ap_22}. A similar MFEM approach is also used for other hyperbolic systems that can be rewritten as a port-Hamiltonian system, such as the shallow water equations \cite{matignon_2025_pfem_shallow_water}.
Extending the error estimate from \cite{egger_ap_22} without making use of the nudging term gives an error bound that blows up exponentially in time. However, using  that the nudging term drives the discrete observer towards the original system's state, we are able to prove that errors are at most amplified by a constant factor that is independent of time, instead.

  Our main result is given by Corollary \ref{col:unif_conv}, a uniform-in-time error bound that compares the original system state with the numerical solution of the observer system. The bound is the sum of three parts: the first part decays exponentially in time and is proportional to the size of the initial error of the observer. The second part is the discretization error, which is proportional to the grid sizes in space and time, and the third part is proportional to the size of the measurement errors as well as the nudging parameter. The proportionality constants for the second and third terms are independent of time. This is the best estimate one can expect, up to improvements of the constants, because the scheme is only consistent at first order in space and time, and in a single step of the scheme, measurement errors are already multiplied by the nudging parameter.

 The proof of Corollary \ref{col:unif_conv} follows the synchronization proof from \cite{kunkel_obs}, which shows the decay of a modified relative energy. This approach is inspired by \cite{Egger2018} and \cite{zuazua_viscosity}, which use similar techniques to show exponential stability of a damped linear hyperbolic system on networks and exponential stability for a class of scalar nonlinear hyperbolic problems, respectively. Showing the decay of the relative energy  in \cite{kunkel_obs} uses the variational structure of the port-Hamiltonian formulation. The MFEM discretization retains this structure, so many steps of the proof can be easily adapted for the numerical scheme, up to additional consistency errors. This  justifies using the MFEM scheme over conservative schemes for hyperbolic balance laws, such as finite volume schemes. While the relative energy technique is well-suited to handle the present nonlinearity, it is limited to sufficiently smooth solutions. To move beyond smooth solutions, one would have to apply some different stability framework suitable for comparing discontinuous solutions, such as \cite{bressan_stability} or \cite{vasseur_shifts}. However, currently it is unclear to us how to perform this step.

Overall, our results justify using an observer in practice. In comparison to other widely used data assimilation approaches, such as variational or statistical techniques \cite{DA_Book}, observers have the lowest computational cost.
Solving the discrete observer system is typically as costly as solving the original system, because the data is introduced to the PDE on the continuous level and discretization is performed subsequently.

In contrast, variational approaches, such as the 3D-Var and 4D-Var algorithms, are based on minimizing some cost functional that measures the mismatch between model prediction and measured data \cite{DA_Book}.  The minimization typically requires computing gradients repeatedly by solving an adjoint system, which is expensive.
  Statistical approaches use Bayes's formula to estimate the system's state and estimate the uncertainty of the prediction, which is costly. For example, the ensemble Kalman filter (EnKF), going back to \cite{ensemble_kalman_94}, samples an ensemble of system states to estimate the system's statistical properties. In each step, the dynamics of the underlying equations must be solved for each state in the ensemble, making the EnKF many times more expensive than an observer.

 The remainder of this work is structured as follows. Section \ref{sec:rel_work} gives an overview of some related works. In  Section \ref{summary} we introduce the barotropic Euler equations, summarize the convergence result of \cite{kunkel_obs} for the continuous observer and discretize the observer system. Section \ref{convergence} summarizes our proof strategy and estimates the modified relative energy. In Section \ref{sec:conv_result} we complete the proof of the error bound of the discrete scheme. Section \ref{numerics} gives some numerical examples illustrating the convergence result. 
 \section{Related Work}
 \label{sec:rel_work}
 
 While observers are well established for data assimilation,  
we are not aware of any other works proving uniform-in-time error bounds for discrete observers for quasilinear hyperbolic balance laws. Here, we give an overview of some related works, focusing on fluids and hyperbolic balance laws, that study the synchronization of observers or prove error bounds for discrete observers. In the following overview, we use “synchronization” to refer to the long-time convergence of both continuous and discrete observer systems towards the original system with respect to time, and “convergence” to refer to error bounds for discrete observers that decrease under mesh refinement.
 
 For the linear wave equation, in \cite{wave_observer} the synchronization of an observer using  field measurements in some subset of the computational domain is shown on the continuous level. Then, \cite{wave_observer_numerical}  proves the convergence of a numerical discretization of the proposed observer. Additionally, \cite{wave_observer_numerical}  extends the observer to other linear wave-like equations. In \cite{kinetic_nudging} the synchronization of a Luenberger  observer for a kinetic formulation of the Saint-Venant system,  using measurements of the full system state, is shown. Some extensions, such as measurements only on some subset of the domain or noisy measurements are also discussed.  For a  linearized water waves model,  \cite{lissy_perrin_2025} investigates the synchronization of a Luenberger observer using only measurements of the free surface height, in some subset of the spatial domain, and shows convergence of the low-frequency modes of the observer after discretization.
 
  Further examples of observers for fluids include \cite{Jolly2019} where synchronization for the surface quasi-geostrophic equation using blurred-in-time measurements of the full system state is shown on the continuous level, and \cite{feireisl_3d} which proves the synchronization of a nudging approach for 3D Rayleigh-Bénard flow in a weak solution framework. 
  
On the discrete side, some examples include  a computational study of 2D Rayleigh-Bénard flow with coarse-grained velocity data \cite{Farhat2018},  a proof of the convergence and long-time accuracy of a $C^0$ interior penalty method for the Cahn-Hilliard equation using coarse measurements \cite{DIEGEL2022127042} and \cite{discrete_navier_stokes} proving uniform-in-time error bounds for a mixed finite-element semidiscretization in space of the incompressible Navier-Stokes equation with velocity measurements on a coarse spatial mesh. 

A key part of the analysis in \cite{discrete_navier_stokes} is that the viscosity damps the fine scales of the observer system, such that one does not need measurement data on an arbitrarily fine spatial mesh to show exact synchronization of the discrete observer.
In contrast, one cannot expect the same behavior for the barotropic Euler equations. Because there is no viscosity, errors made on  fine scales persist. This also poses an additional challenge for discretization, because the discretization error occurs at the fine scales. Indeed, \cite{titi_collins_2025} shows that for non-dissipative systems, such as the Korteweg-de Vries equation, with only sparse measurement data, one should not generally expect  exact synchronization with the original system state, when the high frequency modes are not captured by the measurements.  For the linearized water waves model, \cite{lissy_perrin_2025} observes that high-frequency modes reduce the speed at which the observer system achieves synchronization, and can even reduce the rate from exponential to polynomial. Throughout this work, we assume that we have measurements of the velocity at arbitrarily fine resolution, such that the data resolves even the high-frequency modes, avoiding these issues. In practice, this is unrealistic. While data may be available at high resolution (e.g. PIV), ultimately one only has discrete sensor locations that cannot capture all frequencies. A natural extension of the nudging approach to this situation is given by using interpolated data for nudging \cite{AOT_algo}.  The interpolation error then gives an additional contribution to the bound from Corollary \ref{col:unif_conv}, which does not decay over  time, because the Euler system has no damping for the frequencies which are not resolved by the measurements, see Remark \ref{remark:coarse_data} for a more quantitative explanation.

In \cite{zuazua_viscosity}, the semi-discretization in time of linear, exponentially stable systems is studied using energy techniques. As justified by \cite{wave_observer}, one can view an observer for the linear wave equation as a special case of such an exponentially stable system. Similar to the results in \cite{discrete_navier_stokes}, the analysis in \cite{zuazua_viscosity} shows  that one has to introduce suitable numerical viscosity to dissipate the fine scales of such exponentially stable systems and show convergence of the discretization.

While we consider a quasilinear equation in this work, due to the smallness conditions we require, the analysis shares similarities with the linear case. From \cite{kunkel_obs}, we know that the observer system for the barotropic Euler equations is exponentially stable.  The ideas from \cite{zuazua_viscosity} and \cite{discrete_navier_stokes} are reflected in our choice of discretization. In particular, using the implicit Euler method for time integration naturally introduces numerical dissipation into our scheme, which is essential to our convergence proof.

%% file: Sections/summary_cont.tex
\section{Problem Outline and Summary of Previous Results}
\label{summary}
In this section we briefly summarize the assumptions and results from \cite{kunkel_obs}, where the convergence of the continuous distributed observer system towards the original state is shown.

Let $(\hdens,\hvel)$ be a solution of the barotropic Euler equations on a single pipe, with a quadratic friction term and unit cross-sectional area. 
We write the Euler system as a port-Hamiltonian system: 
\begin{align}
	\partial_t \hdens + \partial_x(\hdens\hvel) &= \src_1 \label{eqn:masscons} \\
	\partial_t \hvel + \partial_x \Enthalp(\hdens,\hvel)+\gamma \abs{\hvel}\hvel &= \src_2,\label{eqn:momcons}
\end{align}
for $x \in \xDom := [0,\ell]$ and $t \in \tDom$, with source terms $\src_1,\src_2 \in W^{1,\infty}([0,\infty);W^{1,2}(\xDom)) $. This formulation is equivalent to the conservative form for Lipschitz-solutions and enables the variational arguments from \cite{kunkel_obs} and \cite{egger_ap_22}.
While the source terms $\src_1$ and $\src_2$ are not physically relevant, the subsequent analysis is possible even with source terms, and the source terms make it possible to use manufactured solutions in numerical examples. The solution of system \eqref{eqn:masscons} -- \eqref{eqn:momcons} will serve as a reference solution for the observer system,  which is why we denote its components with hats. We define the enthalpy $\Enthalp(\hdens,\hvel) \coloneqq \frac 12 \hvel^2 + \ppot'(\hdens)$, where $\ppot$ is the pressure potential, connected to the pressure law by $p'(\hdens) = \hdens \ppot''(\hdens)$. For brevity, we define $\hEnthalp = \Enthalp(\hdens,\hvel)$. For subsonic states, one may prescribe one boundary condition at each end of the pipe. We prescribe $\hflow_\partial = \hdens\hvel$ at $x = 0$ and $\hEnthalp_\partial = \hEnthalp$ at $x = \ell$.
The variable that is prescribed at either end of the pipe can be changed, but we always need one boundary condition on $\hflow$ and $\hEnthalp$ each, to eliminate boundary contributions in some estimates, see Remark \ref{rem:boundary_conds} for more details.

For the rest of this work, we assume that we measure only $\hvel$.
 Additionally, real measured data has measurement errors, which we will denote by $\error$.
Given measurements of $\hvel$ on the whole pipe, one wishes to reconstruct the full system state $(\hdens,\hvel)$ on the pipe. We use a Luenberger observer
\begin{align}
	\partial_t \dens + \partial_x(\dens\vel) &= \src_1 \label{eqn:obs_masscons} \\
	\partial_t \vel + \partial_x \Enthalp(\dens,\vel) + \gamma \abs{\vel}\vel &= \src_2 - \mu(\vel -\hvel) + \mu \eta,\label{eqn:obs_momcons}
\end{align}
with some nudging parameter $\mu > 0$, subject to the same boundary conditions as the original system.

In \cite{kunkel_obs}, it is shown that under the following assumptions, the observer system (without measurement error or source terms) converges towards the original system state exponentially in time: 
\begin{enumerate}[leftmargin = 1.5cm,label=\textbf{(A\arabic*)},ref=\textbf{(A\arabic*)}]
	\item There exists a solution of the original system \eqref{eqn:masscons} -- \eqref{eqn:momcons} with the following regularity: 
	\begin{equation}
		\hdens,\hvel \in W^{2,\infty}\left(\tDom;W^{1,1}(\xDom)\right) \cap W^{1,\infty}\left(\tDom;W^{1,\infty}(\xDom)\right).
	\end{equation}
	Further, there exist constants $\underbar{\dens},\overbar{\dens},\tilde \vel>0$, such that the  following uniform bounds are satisfied pointwise:
	\begin{equation}
		0 < \underbar{\dens} \leq \hdens(x,t) \leq \overbar{\dens} \quad \text{and} \quad \abs{\hvel(x,t)} \leq \tilde{\vel},
		\label{eqn:pointwise_unif_bd}
	\end{equation}
	for all $x \in \xDom$ and $t\in\tDom$. \label{ass:unif}
	\item The pressure potential $\ppot:\symbb{R}^+ \to \symbb{R}$ is smooth and strongly convex, i.e. there exists a constant $\underline{C}_{\ppot''}>0$ with
	\begin{equation}
		\ppot''(\rho) \geq \underline{C}_{\ppot''},
	\end{equation} 
	for all densities $\underbar{\dens} \leq \dens \leq \overbar{\dens}$. \label{ass:convexity}
	\item The solution is subsonic. More precisely, we have 
	\begin{equation}
		\rho \ppot''(\rho) \geq 4 \abs{\tilde{\vel}}^2, \quad \text{for all} \quad \underbar{\dens} \leq \dens \leq\overbar{\dens}.
	\end{equation}\label{ass:subsonic}
	\item The solution $(\hdens,\hvel)$ of the original system has sufficiently small time derivatives and both solutions have small velocities, i.e. there exist  sufficiently small constants $C_t,\overbar{\vel} >0$, such that
	\label{ass:smallderivs} 
	\begin{align}
		||\partial_t\hdens||_{\besov{\infty}{\infty}} + ||\partial_t\hvel||_{\besov{\infty}{\infty}} &\leq C_t,\\
		||\vel||_{\besov{\infty}{\infty}}, ||\hvel||_{\besov{\infty}{\infty}} &\leq \overbar{v}.
	\end{align}
	These constants will be further specified in later computations.
\end{enumerate}

 For the remainder of this paper, we abbreviate $||\cdot||_2 \coloneqq || \cdot ||_{L^2(\xDom)} $. 
The central result of \cite{kunkel_obs} is given by
\begin{theorem}(Convergence, c.f. \cite[Lemma 4.6]{kunkel_obs})
	\label{theo:kunkel_conv}
	Under the assumptions \ref{ass:unif} -- \ref{ass:smallderivs}, for any $\mu > 0$, the constants $C_t$ and $\overbar v$ can be chosen small enough, such that there exist constants $C_1,C_2>0$, such that for any $t \in \tDom$, it holds that
	$$
	||\su(t,\cdot) - \hu(t,\cdot)||_2^2 \leq C_1||\su(0,\cdot) - \hu(0,\cdot)||_2^2 \exp(-C_2 t),
	$$
	with $\su = (\dens,\vel)$ and $\hu = (\hdens,\hvel)$. 
\end{theorem}

For the later treatment of the source terms and measurement errors, we add the following additional assumption: 
\begin{enumerate}[leftmargin = 1.5cm,label=\textbf{(A5)},ref=\textbf{(A5)}]
	\item The source terms have the regularity $\src_1,\src_2 \in W^{1,\infty}([0,\infty);W^{1,2}(\xDom)) $. Further, the measurement error has the regularity $\error \in C\left([0,\infty)\times \xDom\right) \cap L^\infty\left([0,\infty)\times\xDom\right)$, i.e. that the measurement error be continuous in space and time and pointwise bounded. 
	\label{ass:meas_err}
\end{enumerate}
The continuity requirement on the measurement error is a technical requirement to simplify the presentation of the estimates, that is, to show that measurement errors will give a contribution  proportional to $||\error||_{\besov{\infty}{2}}^2$. However, the numerical scheme and its analysis  only depend on the linear interpolant of the measurement errors. Thus,  more realistic settings with discrete sensors, each with their own error described by an independent random variable, are a straightforward generalization. The assumption that the measurement error is bounded in space and time is also realistic, because many measurement devices are rated to have both some bounded relative error and bounded maximum error within their operational regime.

We discretize system \eqref{eqn:obs_masscons} -- \eqref{eqn:obs_momcons} using the scheme from  \cite{egger_ap_22}. Let $M\in\symbb{N}$ and set $\xStep = \ell/M$. Let $\symcal{T}_\xStep = \{x_i = i\xStep, 0 \leq i \leq M\}$ be a uniform grid in space. Let 
$$
\symcal{P}_k(\symcal{T}_\xStep) \coloneqq \left\{p:\xDom \mapsto \symbb{R}\: \bigl\lvert \: \forall i = 0,\dotsc,M-1: p|_{[x_i,x_{i+1}]} \text{ is a polynomial of degree k}\right\}.
$$

 We then define the discrete spaces 
$$
Q_\xStep \coloneqq \symcal{P}_0(\symcal{T}_\xStep), \quad \text{and}\quad R_\xStep \coloneqq \{r \in \symcal{P}_1(\symcal{T}_\xStep): r(0) = 0\} \cap H^1(\xDom),
$$
that is, piecewise constant and piecewise linear continuous functions on the grid $\symcal{T}_\xStep$. For later use, we define projection operators. Let $\Pi_\xStep: L^2(\xDom)\to Q_\xStep$ be the $L^2$--projection onto the piecewise constant functions and let $I_\xStep: H^1(\xDom) \to \symcal{P}_1(\symcal{T}_\xStep) \cap H^1(\xDom)$ be the piecewise linear interpolation operator.

 Choose $\tStep > 0$ as the time step size. We define the time steps $t^n = n \tStep$ for all $n\in \symbb{N}$. Let $\ddt \su^n \coloneqq \sfrac{1}{\tStep}(\su^n - \su^{n-1} )$. Lastly, we define the linear interpolants of the velocity and the measurement error function respectively as
 \begin{equation}
 	\label{def:lin_interp_vel}
 	\numhvel^n \coloneqq I_\xStep(\hvel(t^n,\cdot)) \quad \text{and} \quad \numerror^n = I_\xStep(\error(t^n,\cdot)).
 \end{equation}
   Note that the measurement error does not need to be known explicitly for the observer system, because the nudging term requires only the measured data $\hvel + \error$, however we track the term separately to see its influence throughout the following analysis.
     With this, we are ready to state the fully discrete scheme.

\textbf{Numerical Scheme.} Let $\ip{\cdot}{\cdot}$ denote the $L^2$-inner product on $\xDom$.  Let $\numdens^0 \in Q_\xStep$ and $\numflow^0$ with $\numflow^0 -\hflow_\partial(t^0) \in R_\xStep$ be given.
 For $1 \leq n \leq N$, find $\numdens^n \in Q_\xStep$ and  $\numflow^n$ with $\numflow^n -\hflow_\partial(t^n) \in R_\xStep$, such that 
\begin{align}
	\ip{\ddt \numdens^n }{q_\xStep} + \ip{\partial_x \numflow^n }{q_\xStep} =&\ip{\Pi_\xStep (\src_1^n)}{q_\xStep}, \quad \forall q_\xStep \in Q_\xStep,\label{eqn:dis_scheme_1}\\
	\begin{aligned}[b]
	\ip{\ddt \numvel^n }{r_\xStep} -& \ip{\numEnthalp^n}{\partial_x r_\xStep} + \Enthalp_\partial^n r_\xStep\lvert^{x=\ell}_{x=0} 
	\\ &+\gamma\ip{\abs{\numvel^n}\numvel^n}{r_\xStep} + \mu \ip{\numvel^n - (\numhvel^n+ \numerror^n)}{r_\xStep}
	\end{aligned}
	 =& \ip{\src_2^n}{r_\xStep}, \quad \forall r_\xStep \in R_\xStep,\label{eqn:dis_scheme_2}
\end{align}
 with $\numvel^n \coloneqq \frac{\numflow^n }{\numdens^n}$ and $\numEnthalp^n \coloneqq \frac 12 (\numvel^n)^2 + \ppot'(\numdens^n)$. 
 
We define $\Enthalp_\partial^n r_\xStep\lvert^{x=\ell}_{x=0} \coloneqq \Enthalp_\partial(t^n) r_\xStep\lvert_{x=\ell} - \numEnthalp^n r_\xStep\lvert_{x=0}$, which corresponds to weakly enforcing the {boundary} condition on $\Enthalp$ at $x=\ell$. Note that $\numEnthalp^n r_\xStep\lvert_{x=0} = 0$, because $r_\xStep\lvert_{x=0} = 0$ for all $r_\xStep \in R_\xStep$. The boundary condition on $\numflow^n$ at $x=0$ is enforced strongly.

For the numerical analysis, we will compare the numerical solution to the projection of the solution of the  original system in conservative variables, $(\hdens,\hflow)$ with $\hflow \coloneqq \hdens\hvel$. We set 
\begin{align}
	\numhdens^n &\coloneqq \Pi_\xStep(\hdens(t^n,\cdot)), \quad \numhflow^n \coloneqq I_\xStep(\hflow(t^n,\cdot)), \quad \text{and}\\
	\numhu^n &\coloneqq (\numhdens^n,\numhvelNL{n}).\label{def:proj_state}
\end{align}
With these definitions, importantly, $\numhdens^n\numhvel^n \neq \numhflow^n$. Instead, one has $||\numhdens^n \numhvel^n - \numhflow^n||_\infty \leq C \xStep$ with some constant $C$ independent of $n$, which we show in Lemma \ref{lemma:dt_proj_estimates}. Additionally we define $\numhEnthalp^n \coloneqq \frac 12 (\numhvelNL{n})^2 + \ppot'(\numhdens^n)$.

 In \cite{egger_ap_22}, it is shown that, under some additional assumptions, the above scheme converges uniformly in time, at first order in $\tStep$ and $\xStep$. More precisely, one can show that, for all $n$ such that $t^n \leq T$ for some $T \in \symbb{R}$, one gets $||\numu^n - \su^n||_2^2 \leq C(\xStep^2 + \tStep^2)$, with $C \sim \exp(cT)$, coming from a discrete version of Gronwall's inequality. In the convergence proof from \cite{egger_ap_22}, one cannot expect to obtain a better bound than this exponentially large constant.
 
  A naive approach for obtaining an error bound comparing  the discrete observer to the original system could be combining the continuous synchronization result and the error bound for the numerical scheme using the following splitting of the error: 
\begin{equation}
	\label{eqn:bad_splitting}
	||\numu^n - \hu^n ||_2 \leq ||\numu^n - \su^n ||_2 + ||\su^n - \hu^n||_2.
\end{equation}
The first term can be estimated using the error estimate of the numerical scheme (after some simple modifications to include the nudging term of the observer system). The second term can be controlled using the synchronization of the continuous observer system, Theorem~\ref{theo:kunkel_conv}.  Thus, one gets an overall error estimate of $||\numu^n - \hu^n ||_2^2 \lesssim  \left(\xStep^2 + \tStep^2\right)\exp(C_1 T) + \exp(-C_2T)$ at time $t = T$. However, this estimate is not uniform-in-time, which leads to much more restrictive bounds on the grid sizes for long-time simulations. To guarantee $||\numu^n - \hu^n ||_2^2 \leq \epsilon$ for some prescribed error tolerance $\epsilon > 0$ one requires $T \gtrsim -\log(\epsilon)/C_2$ and $(\xStep^2 + \tStep^2) \lesssim \epsilon^{1 + \frac{C_1}{C_2}}$. From the proofs of \cite{egger_ap_22} and \cite{kunkel_obs}, one finds that $C_1 \gg C_2$, such that the estimate scales poorly for $\epsilon\to 0$. 

We improve the error bound by using a different splitting of the error. By comparing the solution of the discrete observer directly to the projection of the original system state instead, we can make use of the nudging term and obtain an error bound that does not deteriorate over time. This leads to the following splitting:
\begin{equation}
	||\hu^n - \numu^n||_2 \leq ||\hu^n - \numhu^n ||_2 + ||\numhu^n - \numu^n||_2.	\label{eqn:dis_error_split}
\end{equation}
Due to Assumption \ref{ass:unif}, and in particular, since $\hu$ is Lipschitz uniformly in time, the projection error can be bounded by $||\hu^n - \numhu^n ||_2 \leq C\xStep$ for some $C>0$ independent of $T$.
 For the second term, the nudging term makes it possible to prove an error estimate with a constant independent time.

%% file: Sections/conv_proof.tex
\section{Convergence of the discrete observer.}
\label{convergence}
In this section, we will prove the convergence of the discrete observer system.
Our final error estimate is given by the following corollary, which we will prove later:
{
	\renewcommand{\thedefinition}{\ref{col:unif_conv}}
\begin{corollary}
	Let $\hu$ be a solution of system \eqref{eqn:masscons} -- \eqref{eqn:momcons}  and $\numu^n$ be a solution of \eqref{eqn:dis_scheme_1} -- \eqref{eqn:dis_scheme_2}. Further,  let Assumptions \ref{ass:unif} -- \ref{ass:smallderivs}, \ref{ass:meas_err} and \ref{ass:discrete} hold. Let  $\alpha, \tStep_0 \in \symbb{R}$ be chosen small enough, such that $C_1C_0^{-1} >(\alpha + \tStep_0) C_{2}c_0^{-1}$, with the constants as introduced in Theorem \ref{thm:unif_conv}.
	Then, for all $0 < \tStep \leq \tStep_0$ and $0 <\xStep \leq \xStep_0$  with $\tStep_0$ and $\xStep_0$ sufficiently small, for any $n \in \symbb{N}$, it holds that 
	\begin{equation}
		\begin{aligned}
			\frac 12 c_0 ||\numu^n -\hu^n||_2^2 &\leq 2 C_0 \left(\prod_{j=1}^n {w}_j^{-1}\right) ||\numu^0 -\numhu^0||_2^2 + \tilde C(\tStep^2 + \xStep^2) + \mu^2 C||\error||_{\besov{\infty}{2}}^2,
		\end{aligned}
	\end{equation}
	 In particular, the weights $ w_n$ behave as $1 + \symcal{O}(\tStep)$, such that their product represents a discrete approximation to exponential decay.
	 Further, the constants  $C_0, C$ and $\tilde C$ independent of $n$, $\tStep$ and $\xStep$.
\end{corollary}\noindent
Corollary \ref{col:unif_conv} shows that the discrete observer synchronizes with the original system solution, up to some base level of error due to measurement errors and the discretization errors.
	\addtocounter{definition}{-1}
} 
%
%

We use a relative energy framework, as presented in \cite{kunkel_obs}, and adapt the estimates therein to the discrete setting. In order to motivate the more technical discrete computations, we  summarize the strategy from \cite{kunkel_obs} in the following section.

The system \eqref{eqn:masscons} -- \eqref{eqn:momcons} is equipped with the energy 
\begin{equation}
	\label{eqn:energy}
	\hamil(\hdens,\hvel) = \int_0^\ell \frac 12 \hvel^2 + \ppot(\hdens) \dd{x}
\end{equation}
and the state variables $(\hdens,\hvel)$ are connected to the co-state variables $(\hEnthalp,\hflow)$ by the variational derivatives of this energy, that is, $(\hEnthalp,\hflow) = (\delta_{\dens} \hamil(\hu),\delta_{\vel} \hamil(\hu))$. 
We use the energy to define the relative energy between two functions $u$ and $\hu$, given by 
\begin{equation}
	\label{eqn:rel_energy_def}
	\hamil(\su,\hu) \coloneqq \hamil(\su) - \hamil(\hu) - \ip{\hamil'(\hu)}{\su-\hu},
\end{equation}
 The relative energy amounts to the difference of the energy of state $\su$ and the linear Taylor expansion of the energy around $\hu$. Under the uniform bounds from assumption \ref{ass:unif} and the convexity of the pressure potential \ref{ass:convexity}, the relative energy can be used to measure the distance between two functions. To make this more precise, we recall \cite[Lemma 2.5]{egger_ap_22} and \cite[Lemma 4.3]{egger_ap_22}:
\begin{lemma}[Norm equivalence]
	\label{lemma:norm_equiv}
	 Let $\su,\hu \in L^\infty(\xDom)^2$ satisfy the bounds from equations \eqref{eqn:pointwise_unif_bd}. Further, let Assumption \ref{ass:convexity} hold. Then , there exist $c_0,C_0>0$ with
	\begin{equation}
		c_0 ||\su - \hu||_2^2 \leq \hamil(\su,\hu) \leq C_0 ||\su - \hu||_2^2.
	\end{equation}
	The constants $c_0$ and $C_0$ depend only on the bounds in the assumptions.
\end{lemma}
To control the norm $||\su-\hu||_2$, one can thus bound the relative energy. This is done using a Gronwall-inequality type argument, which requires estimates for the time derivative of the relative energy.

Now, let $\hu$ be a solution of system \eqref{eqn:masscons} -- \eqref{eqn:momcons} and let $\su $ be a solution to \eqref{eqn:obs_masscons} -- \eqref{eqn:obs_momcons}. Estimating the time derivative of the relative energy, one only obtains dissipation in the difference of velocities, i.e.
$$
\partial_t \hamil(\su,\hu) \leq -\frac 12 \mu \underbar{\dens} ||\vel - \hvel||_2^2 + C ||\su - \hu||_2^2
$$
for some constant $C$ independent of $\mu$, but depending on the velocity bounds $\overbar{v}$ and the time derivatives of the state $\hu$, c.f. Assumption \ref{ass:smallderivs}. Choosing these bounds small enough, $C$ can be made arbitrarily small. However, because there is no dissipation in $||\dens - \hdens||_2$, one cannot show exponential decay of the relative energy.

Instead, the idea is to find a modified relative energy for which dissipation in both variables can be shown. One ansatz for constructing such a modified relative energy is to find a functional $\contG$ and some small parameter $\delta\in\symbb{R}^+$ with the following properties:
\begin{enumerate}[leftmargin = 1.5cm,label=\textbf{(R\arabic*)},ref=\textbf{(R\arabic*)},start=1]
	\item For $\delta$ sufficiently small, one has the norm equivalence $$||\su - \hu||_2^2 \lesssim\hamil(\su,\hu) + \delta\contG(\su,\hu)\lesssim ||\su - \hu||_2^2.$$
	\label{req:norm_equiv}
	\item For any choice of $\mu$, there exist sufficiently small choices of $\delta$ and the bounds in Assumption \ref{ass:smallderivs} such that 
	\begin{equation}
		\partial_t \left(\hamil(\su,\hu) + \delta\contG(\su,\hu)\right) \leq -C\left(\hamil(\su,\hu) + \delta\contG(\su,\hu)\right),
	\end{equation}
	with some constant $C>0$ depending on the bounds in Assumptions \ref{ass:unif} -- \ref{ass:smallderivs}.
	\label{req:dissipative}
\end{enumerate}
Applying Gronwall's lemma then proves exponential-in-time convergence of the observer state towards the original system state as in Theorem \ref{theo:kunkel_conv}. In \cite{kunkel_obs}, a suitable functional $\contG$ is given, and the above properties are verified.

In this work we extend these arguments to the discrete observer \eqref{eqn:dis_scheme_1}--\eqref{eqn:dis_scheme_2}, to control the discrete error from the error splitting in \eqref{eqn:dis_error_split}.  For the discrete case, we cannot use the functional $\contG$ as in \cite{kunkel_obs}. We define a discrete functional $\disG$, inspired by $\contG$, instead. Then, we estimate the discrete time derivatives of $\hamil$ and $\disG$.

\subsection{Additional assumptions}
In addition to the assumptions \ref{ass:unif} -- \ref{ass:smallderivs}, we require an assumption on the iterates $\numu^n$ of the numerical scheme \eqref{eqn:dis_scheme_1} -- \eqref{eqn:dis_scheme_2}. To show the convergence of the observer system towards the original system state, the continuous analysis requires that the observer system satisfies similar uniform bounds as well. In the continuous case, provided the original system state is uniformly bounded and sufficiently regular, in \cite[Remark 4.1]{kunkel_obs}  it is shown that a solution of the observer system satisfying the assumptions exists.

Currently, we are unable to show that similar uniform bounds will hold for the iterates of the numerical scheme (independent of the discretization parameters $\xStep$ and $\tStep$), however, after computing a numerical solution of the observer system, it is straightforward to check the following assumption in an a-posteriori manner. 
\begin{enumerate}[leftmargin = 1.5cm,label=\textbf{(A1h)},ref=\textbf{(A1h)}]
	\item There exists a solution $(\numdens^n,\numflow^n)_{n \in \symbb{N}}$ of the discrete scheme \eqref{eqn:dis_scheme_1} -- \eqref{eqn:dis_scheme_2} such that for all $n \in \symbb{N}$ and $x \in \xDom$, we have 
	\begin{equation}
			0 < \underbar{\dens} \leq \numdens^n(x) \leq \overbar{\dens} \quad \text{and} \quad \abs{\numvel^n(x)} \leq \overbar{\vel},
	\end{equation}
	where $\numvel^n = \sfrac{\numflow^n}{\numdens^n}$. \label{ass:discrete}
\end{enumerate}

With these key assumptions, we are now ready to analyze the convergence of the discrete observer.
\subsection{Projection errors.}
Before we start estimating the discrete time derivative of the relative energy, we summarize some properties of the projection operators. 

We use the following lemma to estimate projection errors:
\begin{lemma}[Projection Errors, c.f. {\cite[Lemma 4.1]{egger_ap_22}} ]
	\label{lemma:proj_estimates}
	Let $z \in W^{1,p}(\xDom)$ with $1\leq p \leq \infty$. It holds that 
	$$
	\partial_x (I_\xStep z) = \Pi_\xStep(\partial_x z).
	$$
	Further, the projections preserve $||\cdot||_\infty$--bounds:
	$$||\Pi_\xStep z||_\infty \leq ||z||_\infty, \quad \text{and}\quad ||I_\xStep z||_\infty \leq ||z||_\infty.$$	
	Lastly, there exists a constant $c>0$ such that for any $1 \leq p \leq \infty$, the following estimates hold:
		$$||z - \Pi_\xStep z||_p \leq c\xStep||\partial_x z||_p, \quad \text{and}\quad ||z - I_\xStep z||_p \leq c \xStep||\partial_x z||_p.$$
	The constant $c$ is independent of $z$, $p$ and $\xStep$.
\end{lemma}
Using these projection error estimates and the positive lower bound on $\hdens$, one can prove the following estimates for $\numhu^n = (\numhdens^n, \numhvelNL{n})$:
\begin{lemma}
	\label{lemma:dt_proj_estimates}
	Let $\numhu^n$ be the projected original system state given by \eqref{def:proj_state}. Further, let Assumption \ref{ass:unif} hold. Let $* \in \{2,\infty\}$. Then, there exists a constant $C_*>0$, with
	\begin{equation}
		||\numhvelNL{n} - \hvel^n||_* \leq C_* \xStep,
		\label{eqn:numhvel_bound}
	\end{equation}
	with $C_*$  independent of $n$ and $\xStep$.
	Further, there exists a constant $\tilde C_{*}>0$, such that 
		\begin{equation}
		||\ddt(\numhvelNL{n}) - \partial_t\hvel^n||_* \leq \tilde C_* (\xStep + \tStep),
		\label{eqn:numhvel_ddt_bound}
	\end{equation}
	with $\tilde C_{*}$ independent of $n$, $\xStep$ and $\tStep$.
\end{lemma}
\begin{remark}
	\label{remark:dt_proj_estimates}
	Lemma \ref{lemma:dt_proj_estimates} means that the projected original system state satisfies both bounds from Assumption \ref{ass:smallderivs} approximately, in the following sense:
	$$
	||\numhvelNL{n}||_\infty \leq \overbar{\vel} + C_\infty \xStep \quad \text{and} \quad ||\ddt(\numhvelNL{n})||_\infty  \leq ||\partial_t \hvel||_{\besov{\infty}{\infty}} + \tilde C_\infty (\tStep + \xStep).
	$$
	Tracking $\xStep$- and $\tStep$- dependent bounds for  $||\numhvelNL{n}||_\infty $ and $||\ddt(\numhvelNL{n})||_\infty$ is cumbersome for the further estimates in this work. However, for all $0 < \tStep\leq \tStep_0$ and $0<\xStep \leq \xStep_0$,  with $\tStep_0$ and $\xStep_0$ small enough, the projected original system state $\numhu^n$ satisfies bounds independent of $\tStep$ and $\xStep$. These bounds can be chosen arbitrarily close to the bounds in \ref{ass:smallderivs}, and thus, we will use the same names for these constants.
\end{remark}

\begin{proof}[Proof of Lemma \ref{lemma:dt_proj_estimates}]
	We start by proving estimate \eqref{eqn:numhvel_bound}.
	 Using $\numhvelNL{n} - \hvel^n = \frac{1}{\hdens^n}(\numhflow^n - \hflow^n ) + {\frac{\numhflow^n}{\hdens^n \numhdens^n} (\hdens^n - \numhdens^n)}$ gives
	\begin{equation}
		||\numhvelNL{n} - \hvel ||_* \leq \frac{1}{\underbar{\dens}}||\hflow^n - \numhflow^n||_* + \frac{\overbar{\dens}\overbar{\vel}}{\underbar{\dens}^2}||\hdens^n - \numhdens^n||_* \leq c \xStep \left(\frac{1}{\underbar\dens}||\partial_x \hflow^n||_* + \frac{\overbar{\dens}\overbar{\vel}}{\underbar{\dens}^2}||\partial_x\hdens^n||_*\right),
	\end{equation}
	where the first inequality relies on the uniform bounds in \ref{ass:unif} being preserved by the projections, and the second inequality is just the projection error estimates.
	Finally, we have $||\partial_x\hdens^n||_* \leq ||\partial_x\hdens||_{\besov{\infty}{*}} $ and $||\partial_x\hflow^n||_* \leq ||\partial_x\hflow||_{\besov{\infty}{*}} $, such that the constant can be made independent of $n$.
	
	We move on to estimate \eqref{eqn:numhvel_ddt_bound}. Using the triangle inequality gives
	\begin{equation}
		\label{eqn:proof_numhvel_ddt_1}
	||\partial_t\hvel^n - \ddt(\numhvelNL{n})||_* \leq ||\partial_t\hvel^n - \partial_t(\numhvelNL{n})||_* + ||\partial_t(\numhvelNL{n}) - \ddt(\numhvelNL{n})||_*.
	\end{equation}
	We consider both contributions separately. Using the projection error estimates, and the chain rule, for the first term we have 
	\begin{equation}
		\begin{aligned}
			||\partial_t\hvel^n - \partial_t(\numhvelNL{n})||_* &\leq ||\partial_t(\sfrac{\hflow^n}{\hdens^n}) - \partial_t(\sfrac{\numhflow^n}{\hdens^n})||_* + ||\partial_t(\sfrac{\numhflow^n}{\hdens^n})- \partial_t(\numhvelNL{n})||_* \\
			&\leq \frac{1}{\underbar{\dens}}||\partial_t \hflow^n - \partial_t \numhflow^n||_* + \frac{1}{\underbar{\dens}^2}||\partial_t \hdens^n||_\infty||\hflow^n -\numhflow^n||_* + \frac{1}{\underbar{\dens}^2}||\partial_t \numhflow||_\infty||\hdens^n -\numhdens^n||_* \\
			&\quad + ||\numhflow^n||_\infty\left(\frac{1}{\underbar{\dens}^2}||\partial_t\hdens^n - \partial_t\numhdens^n||_* + \frac{\overbar{\dens}^2}{\underbar{\dens}^4}||\partial_t\hdens^n||_\infty||\hdens - \numhdens||_* \right)\leq C \xStep
		\end{aligned}
		\label{eqn:dthvel_proj_error}
	\end{equation}
	with some constant $C>0$ independent of $\xStep$ and $n$, using that $\hdens,\hflow \in W^{1,\infty}(\tDom;W^{1,\infty}(\xDom))$.
	For the second term in \eqref{eqn:proof_numhvel_ddt_1}, applying Taylor's theorem gives
	\begin{equation}
		\begin{aligned}
			||\partial_t(\numhvelNL{n}) - \ddt(\numhvelNL{n})||_* \leq \frac{\tStep}{2}||\partial_{tt}(\sfrac{\hflow}{\hdens})||_{\besov{\infty}{*}} \leq C \tStep,
		\end{aligned}
		\label{eqn:dthvel_taylor}
	\end{equation}
	with some $C>0$ independent of $\tStep$ and $n$, depending on the constants in assumption \ref{ass:unif}.
\end{proof}

\subsection{Time derivative of $\hamil(\numu,\numhu)$.}
Our estimates follow the strategy from \cite{egger_ap_22}, but modify some key parts. We repeat the computations in detail,  to keep the arguments self-contained. Further, we make some constants more explicit, because estimating too roughly spoils the dissipation needed to show the exponential decrease over time.

Estimates for discrete time derivatives will commonly require the use of Taylor's theorem. For the remainder terms, Taylor's theorem gives intermediate values, for which we introduce the following notation: 
\begin{definition}[Intermediate values]
	Let $x_i$ with $0 \leq i \leq M$ be the grid points in space. Let $f_1,f_2 \in L^2(\xDom)$ with $f_1|_{[x_i,x_{i+1}]},f_2|_{[x_i,x_{i+1}]} \in C([x_i,x_{i+1}])$ for all $i = 0, \dotsc, M-1$. We say a function $f: \symbb{R} \to \symbb{R}$ is an \textit{intermediate value}, denoted by $f \in [f_1,f_2]$, if for all $x \in \xDom$, $f(x)$ lies between $f_1(x)$ and $f_2(x)$.
\end{definition}

We begin by recalling the following result on the time derivative of the relative energy:
\begin{lemma}[Discrete time derivative of  the relative energy]
	\label{lemma:ddt_hamil}
	Let $\numhu^n$ be the projected original system state as defined in \eqref{def:proj_state} and let $\numu^n$ solve system \eqref{eqn:dis_scheme_1} -- \eqref{eqn:dis_scheme_2}. Further, let Assumptions \ref{ass:unif}, \ref{ass:convexity} and \ref{ass:discrete} hold. Then,
	\begin{equation}
		\begin{aligned}
		\ddt \hamil(\numu^n,\numhu^n) =& \ip{\hamil'(\numu^n) - \hamil'(\numhu^{n})}{\ddt(\numu^n - \numhu^n)} \\
		&  + \ip{\hamil'(\numu^n)-\hamil'(\numhu^n)-\hamil''(\numhu^n)(\numu^n-\numhu^n)}{\ddt \numhu^n}  \\ &+\ip{\hamil''(\numhu^n)(\numu^n-\numhu^n)}{\ddt\numhu^n}  -\ip{\ddt\hamil'(\numhu^n)}{\numu^n-\numhu^n} \\ 
			&+\frac 12 \tau \biggl(-\ip{\hamil''(\numu^*)\ddt\numu^n}{\ddt\numu^n} +\ip{\hamil''(\numhu^*)\ddt\numhu^n}{\ddt\numhu^n}  +2\ip{\ddt\hamil'(\numhu^n)}{\ddt(\numu^n-\numhu^n)}\biggr) 
		\\ 
		&\eqqcolon (I) + (II) + (III) + (IV),
	\end{aligned}
	\label{eqn:dtH_estimate}
	\end{equation}
	with  $\numu^* \in [\numu^{n-1},\numu^n]$ and $\numhu^* \in [\numhu^{n-1},\numhu^n]$.
	
	Terms that are analogous to $(I)$ and $(II)$, but with $\ddt$ replaced by  $\partial_t$, also appear in the continuous time derivative of $\hamil(\su,\hu)$. The terms $(III)$ and $(IV)$ are perturbations due to the discretization.
\end{lemma}
\begin{proof}This result is shown in the proof of \cite[Lemma 4.4]{egger_ap_22}, as an intermediate step.\end{proof}

Next, we estimate the terms of the previous lemma individually. We begin by estimating the contributions from the discrete perturbation terms.
 For future convenience, we first recall the estimate from \cite[Lemma 4.3]{egger_ap_22}, as it will be commonly used in the following computations.
\begin{lemma}
	\label{lemma:bounded}
	Let assumptions \ref{ass:unif} -- \ref{ass:convexity} hold. Then, there exists a constant $C>0$ such that for any $x \in L^\infty(\xDom)$ and $y \in L^2(\xDom)$, one has 
	\begin{equation}
		\ip{(\hamil''(\su) - \hamil''(\hu)) x}{y} \leq C||\su - \hu||_2 ||x||_\infty ||y||_2.
	\end{equation}
	The constant C depends only on the bounds from the assumptions.
	
	Further, $\hamil''(\hu)$ is uniformly convex for any $\hu\in\symbb{R}^2$ satisfying the bounds from Assumption \ref{ass:unif}. More precisely, we have 
	\begin{equation}
		z^T \hamil''(\hu) z \geq \underline {C} z^2, \quad \forall z\in \symbb{R}^2,
	\end{equation}
	with $\underline{C} = \min(\frac{\underline C_{\ppot''}}{4},\frac{\underline\dens}{4})$.
\end{lemma}

\begin{lemma}[Estimate of (IV) from Lemma \ref{lemma:ddt_hamil}]
	\label{lemma:dthp2}	
	Under the assumptions of Lemma \ref{lemma:ddt_hamil}, there exist $\tStep_0,\xStep_0\in \symbb{R}$ such that for all $0 < \tStep \leq \tStep_0$ and $0 < \xStep \leq \xStep_0$, there exist $C_1,C_2,C_3 > 0$, such that, for any $\alpha \in \symbb{R}^+$ it holds that 
	\begin{equation}
			\begin{aligned}
			\frac 12 \tStep &\biggl(-\ip{\hamil''(\numu^*)\ddt\numu^n}{\ddt\numu^n} +\ip{\hamil''(\numhu^*)\ddt\numhu^n}{\ddt\numhu^n}  +2\ip{\ddt\hamil'(\numhu^n)}{\ddt(\numu^n-\numhu^n)}\biggr) \\
			&\leq {\alpha} \left(||\numhu^{n-1}-\numu^{n-1}||_2^2 + ||\numhu^{n}-\numu^{n}||_2^2\right) + C_1 \tStep \left(||\numhu^{n-1}-\numu^{n-1}||_2^2 + ||\numhu^{n}-\numu^{n}||_2^2\right) \\
			&\quad + \frac{C_2}{\alpha} \tStep^2  + C_3 \tStep^2- \frac 12\tStep\underline{C}||\ddt(\numu^n-\numhu^n)||_2^2.
		\end{aligned}
		\label{eqn:numdiss}
	\end{equation}

	The constants $C_1,C_2$ and $C_3$ only depend on the bounds in assumptions \ref{ass:unif} and \ref{ass:convexity} and are independent of $n$, $\xStep$ and $\tStep$.
\end{lemma}
\begin{proof}
	We begin by rearranging $(IV)$ to obtain a numerical dissipation term. By inserting zeros suitably, we get
	\begin{equation}
		\begin{aligned}
			&\frac 12 \tStep \biggl(-\ip{\hamil''(\numu^*)\ddt\numu^n}{\ddt\numu^n} +\ip{\hamil''(\numhu^*)\ddt\numhu^n}{\ddt\numhu^n}  +2\ip{\ddt\hamil'(\numhu^n)}{\ddt(\numu^n-\numhu^n)}\biggr)\\
			&\quad =  -\frac 12\tStep \ip{\hamil''(\numu^*)\ddt(\numu^n-\numhu^n)}{\ddt(\numu^n-\numhu^n)} \\
			&\quad\quad +\frac 12\tStep \left(\ip{\hamil''(\numhu^*)\ddt\numhu^n}{\ddt\numhu^n} - \ip{\hamil''(\numu^*)\ddt\numhu^n}{\ddt\numhu^n}\right)\\
			&\quad\quad+	\frac 12\tStep\left(2\ip{\ddt\hamil'(\numhu^n)}{\ddt(\numu^n-\numhu^n)} - 2\ip{\hamil''(\numu^*)\ddt\left(\numu^n-\numhu^n \right)}{\ddt\numhu^n}\right)\\
			&\eqqcolon (A) + (B) + (C).
		\end{aligned}
		\label{eqn:numdiss_split_terms}
	\end{equation}
	We estimate $(A)$ by using the uniform convexity of the energy (c.f. Lemma \ref{lemma:bounded}):
	\begin{equation}
		 (A) = -\frac 12\tStep \ip{\hamil''(\numu^*)\ddt(\numu^n-\numhu^n)}{\ddt(\numu^n-\numhu^n)} \leq- \tStep\underline{C}||\ddt(\numu^n-\numhu^n)||_2^2.
		 \label{eqn:num_diss_est}
	\end{equation}
	This is the desired dissipation term. We will use it to absorb the contribution of $(C)$.
	
	Then, to estimate $(B)$, we use Lemma \ref{lemma:bounded}, giving 
	\begin{align}
		\label{eq:lem6_est_b}
		(B) = 	\frac 12 \tStep \ip{(\hamil''(\numhu^*)-\hamil''(\numu^*))\ddt\numhu^n}{\ddt\numhu^n} \leq \frac 12 C\tStep||\numhu^*-\numu^*||_2||\ddt\numhu^n||_2||\ddt\numhu^n||_\infty.
	\end{align}
	Because $\numhu^*\in[\numhu^{n-1},\numhu^n]$ and $\numu^* \in [\numu^{n-1},\numu^n]$ are intermediate values, it holds that 
	\begin{equation}
		||\numhu^*-\numu^*||_2 \leq ||\numhu^{n-1}-\numu^{n-1}||_2 + ||\numhu^n-\numhu^{n-1}||_2 + ||\numhu^n-\numu^n||_2.
		\label{eqn:intermediate_value}
	\end{equation}
	Further, $||\numhu^n-\numhu^{n-1}||_2 = \tStep||\ddt \numhu^n||_2$ and from the definition of the projection operators, assumption \ref{ass:unif} and Lemma \ref{lemma:dt_proj_estimates} we have $||\ddt \numhu^n||_2 \leq ||\partial_t \hu||_{\besov{\infty}{2}} + \tilde C(\tStep + \xStep) \leq C$, for  all $0<\tStep \leq \tStep_0$ and $0<\xStep \leq \xStep_0$ sufficiently small. The constant $C$ is independent of $\tStep$ and $\xStep$. Similarly, we have $||\ddt \numhu^n||_\infty \leq C$ with $C$ independent of $\tStep$ and $\xStep$.
	Then, applying weighted Young's inequality with $\alpha \in \symbb{R}^+$ gives
	\begin{align}
		\frac 12 C\tStep||\numhu^*-\numu^*||_2||\ddt\numhu^n||_2||\ddt\numhu^n||_\infty \leq& \,{\alpha} \left(||\numhu^{n-1}-\numu^{n-1}||_2^2 + ||\numhu^{n}-\numu^{n}||_2^2\right) + \frac{1}{\alpha} \tStep^2 C,
	\end{align}
	with $C$ some constant depending on the bounds from assumption \ref{ass:unif} and \ref{ass:convexity}, which is independent of $\tStep$ and $\xStep$.
	
	Lastly, to estimate $(C)$, we use Taylor's theorem to expand $$\ip{\ddt\hamil'(\numhu^n)}{\ddt(\numu^n-\numhu^n)} =\ip{\hamil''(\numhu^{**})\ddt\numhu^n}{\ddt(\numu^n-\numhu^n)} ,$$
	with $\numhu^{**}\in [\numhu^{n-1},\numhu^n]$ some intermediate value.
	Thus, using the symmetry of $\hamil''$ and applying Lemma \ref{lemma:bounded}, we have 
	\begin{equation}
	(C) = 	-\tStep\ip{(\hamil''(\numu^*)-\hamil''(\numhu^{**}))\ddt\numhu^n}{\ddt(\numu^n-\numhu^n)} \leq C\tStep||\numu^*-\numhu^{**}||_2||\ddt(\numu^n-\numhu^n)||_2||\ddt\numhu^n||_\infty.
	\end{equation}
	We must pay special attention to $||\ddt(\numu^n-\numhu^n)||_2$, because our assumptions contain no uniform bounds for terms of this type.
	 By applying weighted Young's inequality with $\beta \in \symbb{R}^+$, we get
	\begin{equation}
		\label{eqn:proof_est_IV_1}
		C\tStep||\numu^*-\numhu^{**}||_2||\ddt(\numu^n-\numhu^n)||_2||\ddt\numhu^n||_\infty\leq \tStep\beta||\ddt(\numu^n-\numhu^n)||_2^2 + \frac{1}{\beta} C^2\tStep||\numu^*-\numhu^{**}||_2^2||\ddt\numhu^n||_\infty^2.
	\end{equation}
	We set $\beta =  \sfrac{\underline{C}}{2}$. Then, the first term of the above estimate can be absorbed into \eqref{eqn:num_diss_est}.
	To bound the second term of \eqref{eqn:proof_est_IV_1}, we split $||\numu^* - \numhu^{**}||_2$ using \eqref{eqn:intermediate_value}. This gives
	\begin{equation}
		\begin{aligned}
		\frac{2}{\underline{C}} C^2\tStep||\numu^*-\numhu^{**}||_2^2||\ddt\numhu^n||_\infty^2 &\leq \frac{6}{\underline{C}} C^2\tStep \left(||\numhu^{n-1}-\numu^{n-1}||_2^2 + ||\numhu^{n}-\numu^{n}||_2^2 + \tStep^2||\ddt\numhu^n||_2^2\right)||\ddt\numhu^n||_\infty^2
		\\ &\leq C\tStep(||\numhu^{n-1}-\numu^{n-1}||_2^2 + ||\numhu^{n}-\numu^{n}||_2^2) + C \tStep^2,
		\end{aligned}		 
	\end{equation}
	with $C$ some constant independent of $\tStep$, $\xStep$ and $T$. We have absorbed the $\tStep^3$-term into the $\tStep^2$ terms, again, for $0<\tStep \leq \tStep_0$ sufficiently small.
	Thus, overall, we get the estimate 
	\begin{equation}
		(A) + (C) \leq - \frac 12\tStep\underline{C}||\ddt(\numu^n-\numhu^n)||_2^2 + C\tStep(||\numhu^{n-1}-\numu^{n-1}||_2^2 + ||\numhu^{n}-\numu^{n}||_2^2) + C \tStep^2.
	\end{equation}
	Combining this estimate with equation \eqref{eq:lem6_est_b} and renaming constants gives the statement of the lemma.
\end{proof}
\begin{remark}
	The previous lemma means that the perturbation term (IV) from Lemma \ref{lemma:ddt_hamil} gives three types of contributions: 
	\begin{itemize}
		\item contributions that are proportional to $||\numu - \numhu||_2^2$, but can be made arbitrarily small by choosing $\alpha$ small enough,
		\item contributions that are proportional to $||\numu - \numhu||_2^2$ and  $\tStep$, and
		\item contributions at order $\symcal{O}(\tStep^2)$.
	\end{itemize}
	The first type of contributions does not spoil the constants one would get when estimating the time derivative of the relative energy in the continuous case, and for sufficiently small $\tau$ the second type of contributions will not spoil the exponential decay of the relative energy. Finally, the last type contributes to the base level of error the observer can achieve. Choosing small $\alpha$ gives worse constants on these terms.
	  We will show similar results for the remaining terms of the discrete time derivative $\ddt\hamil(\numu^n,\numhu^n)$, too.
\end{remark}

\begin{lemma}[Estimate of (III) from Lemma \ref{lemma:ddt_hamil}]
	Under the assumptions of Lemma \ref{lemma:ddt_hamil}, there exist $\tStep_0,\xStep_0\in \symbb{R}$ such that for all $0 < \tStep \leq \tStep_0$ and $0 < \xStep \leq \xStep_0$, there exists a $C> 0$ independent of $n$, $\tStep$ and $\xStep$  such that, for any $\alpha \in \symbb{R}^+$ it holds that 
	\begin{equation}
	\ip{\hamil''(\numhu^n)(\numu^n-\numhu^n)}{\ddt\numhu^n}  -\ip{\ddt\hamil'(\numhu^n)}{\numu^n-\numhu^n} \leq \alpha ||\numu^n - \numhu^n||_2^2 + \frac{1}{4\alpha} C \tau^2
	\end{equation}
\end{lemma}
\begin{proof}
We start by using Taylor's theorem to rewrite the term 
$$
\ip{\ddt\hamil'(\numhu^n)}{\numu^n-\numhu^n} = \ip{\hamil''(\numhu^{**})\ddt\numhu^n}{\numu^n-\numhu^n},
$$
with $\numhu^{**} \in [\numhu^{n-1},\numhu^n]$.
Using the fact that $\hamil''$ is symmetric, we thus have 
$$
\ip{\hamil''(\numhu^n)(\numu^n-\numhu^n)}{\ddt\numhu^n}  -\ip{\ddt\hamil'(\numhu^n)}{\numu^n-\numhu^n} = \ip{\left(\hamil''(\numhu^n) - \hamil''(\numhu^{**})\right)\ddt\numhu^n}{\numu^n-\numhu^n}.
$$
Applying Lemma \ref{lemma:bounded}, we obtain
\begin{align}
	\ip{\left(\hamil''(\numhu^n) - \hamil''(\numhu^{**})\right)\ddt\numhu^n}{\numu^n-\numhu^n} \leq C||\numhu^n - \numhu^{**}||_2 ||\numu^n - \numhu^n||_2 ||\ddt\numhu^n||_\infty.
\end{align}
Because $\numhu^{**}\in [\numhu^{n-1},\numhu^n]$, one has $||\numhu^n - \numhu^{**}||_2 \leq \tau ||\ddt\numhu^n||_2$. Due to the bounds from assumption \ref{ass:unif} and Lemma \ref{lemma:dt_proj_estimates}, by applying weighted Young's inequality with $\alpha \in \symbb{R}^+$ we obtain 
\begin{align}
	 ||\numhu^n - \numhu^{**}||_2 ||\numu^n - \numhu^n||_2 ||\ddt\numhu^n||_\infty \leq \alpha ||\numu^n - \numhu^n||_2^2 + \frac{1}{4\alpha} C \tau^2,
\end{align}
where $C$ depends only on properties of the original system state, in particular  $||\partial_t\hu||_{\besov{\infty}{2}}$ and $||\partial_t\hu||_{\besov{\infty}{\infty}}$.

\end{proof}

\begin{lemma}[Estimate of (II) from Lemma \ref{lemma:ddt_hamil}]
	\label{lemma:dth2}
Let $\numhu^n$ be the projected original system state \eqref{def:proj_state} and $\numu^n$ be a solution to system \eqref{eqn:dis_scheme_1} -- \eqref{eqn:dis_scheme_2}. Let Assumptions \ref{ass:unif} and \ref{ass:discrete} hold. Then, the following estimate holds:
\begin{equation}
	\begin{aligned}
		\ip{\hamil'(\numu^n)-\hamil'(\numhu^n)-\hamil''(\numhu^n)(\numu^n-\numhu^n)}{\ddt \numhu^n} \leq \frac 12 \max(1,\overbar{C}_{\ppot'''}) C_t|| \numu^n - \numhu^n||_2^2,
	\end{aligned}
\end{equation}
with $C_t =||\partial_t\numhdens||_{\besov{\infty}{\infty}} +  ||\partial_t\numhvel||_{\besov{\infty}{\infty}}$.
	
\end{lemma}

\begin{proof}
Using the explicit formulas for $\hamil'$ and $\hamil''$, we obtain that 
	\begin{align}
	\hamil'(\numu^n)-\hamil'(\numhu^n)-\hamil''(\numhu^n)(\numu^n-\numhu^n) = 
	\begin{pmatrix}
		\ppot'(\numdens^n,\numhdens^n) + \frac 12 (\numvel^n - \numhvelNL{n})^2 \\
		(\numdens^n-\numhdens^n)(\numvel^n - \numhvelNL{n})
	\end{pmatrix}
\end{align}
with $	\ppot'(\numdens^n,\numhdens^n) \coloneqq \ppot'(\numdens^n) - \ppot'(\numhdens^n) - \ppot''(\numhdens^n)(\numdens^n-\numhdens^n)$. Applying Taylor's theorem, we obtain 
$$
\ppot'(\numdens^n,\numhdens^n) = \frac 12 \ppot'''(\numdens^*)(\numdens^n-\numhdens^n)^2,
$$
with $\numdens^* \in [\numdens^n,\numhdens^n]$. Using the uniform bounds from \ref{ass:unif} and the bounds on the discrete solution \ref{ass:discrete}, we can bound $\ppot'''(\numdens^*) \leq \overbar{C}_{\ppot'''}$. Finally, using Hölder's and Young's inequality we get 
\begin{equation}
\begin{aligned}
&	\ip{\hamil'(\numu^n)-\hamil'(\numhu^n)-\hamil''(\numhu^n)(\numu^n-\numhu^n)}{\ddt \numhu^n}\\
&	\qquad \leq \frac 12 \max(1,\overbar{C}_{\ppot'''}) \left(||\ddt\numhdens^n||_\infty + ||\ddt(\numhvelNL{n})||_\infty\right)|| \numu^n - \numhu^n||_2^2.
\end{aligned}
\end{equation}
The stability of the projection operators implies that $||\ddt\numhdens^n||_\infty \leq ||\partial_t \hdens^n||_\infty$. Further, using Lemma \ref{lemma:dt_proj_estimates} and Remark \ref{remark:dt_proj_estimates} gives 
$$
||\ddt\numhdens^n||_\infty + ||\ddt(\numhvelNL{n})||_\infty \leq C_t,
$$
with $C_t$ from \ref{ass:smallderivs}, for $\tStep$ and $\xStep$ sufficiently small.

\end{proof}
Finally, we need to estimate the term $(I)$ from Lemma \ref{lemma:ddt_hamil}.
This relies on exchanging discrete time derivatives for derivatives in space. For the numerical solution $\numu$,  this uses the discrete scheme \eqref{eqn:dis_scheme_1}--\eqref{eqn:dis_scheme_2}. For the projected original system state $\numhu^n$,  this incurs residuals, which we define as follows:
\begin{align}
	\ip{\res_1^n}{q} &\coloneqq \ip{\ddt\numhdens^n}{q} + \ip{\partial_x\numhflow^n}{q} - \ip{\Pi_\xStep \src_1^n}{q}.\label{eqn:res1}\\
	\ip{\res_2^n}{r}&\coloneqq \ip{\ddt(\numhvelNL{n})}{r} - \ip{\numhEnthalp^n}{\partial_x r} + \gamma\ip{\numhvelNL{n}\abs{\numhvelNL{n}}}{r}  + \hEnthalp^n_\partial r\bigg\lvert_{x=0}^{x=\ell} - \ip{\src_2^n}{r}\label{eqn:res2},
\end{align}
for any $q\in L^2(\xDom)$ and $r \in H^1(\xDom)$. In the following computations, the main technical challenge is estimating these residuals suitably -- as one would expect them to be “small” and decrease with the discretization parameters $\xStep$ and $\tStep$.

\begin{lemma}[Estimate for $(I)$ from Lemma \ref{lemma:ddt_hamil}]
	\label{lemma:rel_diss}
	Let the assumptions of Lemma \ref{lemma:ddt_hamil} hold. Further, let \ref{ass:meas_err} hold.
		Then, there exists a constant $C>0$ depending only on the bounds in \ref{ass:unif}, such that for any $\alpha \in \symbb{R}^+$, it holds that
	\begin{equation}
		\label{eqn:lemma_rel_diss}
		\begin{aligned}
			\ip{\hamil'(\numu^n) - \hamil'(\numhu^{n})}{\ddt(\numu^n - \numhu^n)} \leq& - \frac 12\mu\underbar\dens||\numvel^n-\numhvelNL{n}||_2^2 + \frac 12 \mu \frac{\overbar{v}^2}{\underbar{\dens}} ||\numdens^n - \numhdens^n||_2^2\\
			&+2\gamma \frac{\overbar{v}^3}{\underline{\dens}}||\numdens^n - \numhdens^n||_2^2  + \frac{\mu^2C^2}{\alpha}||\error||^2_{\besov{\infty}{2}} \\
			&-\diss(\numu^n,\numhu^n)+\frac{\mu^2}{4\alpha}C \xStep^2 + \alpha||\numu^n - \numhu^n||_2^2\\
			&- \ip{\res_1^n}{\numEnthalp^n - \numhEnthalp^n} - \ip{\res_2^n}{\numflow^n - \numhflow^n},
		\end{aligned}
	\end{equation}
	
	with the relative dissipation functional $\diss$, defined as 
	\begin{equation}
	\diss(\numu^n,\numhu^n) \coloneqq \frac14\gamma \ip{\numhdens^n(\numvel^n - \numhvelNL{n}) }{(\numvel^n - \numhvelNL{n})(|\numvel^n| + |\numhvelNL{n}|)}.
	\label{eqn:rel_diss}
	\end{equation}
\end{lemma}
Note that the relative dissipation functional is positive, and can be bounded from below by  ${\diss(\numu^n,\numhu^n) \geq c||\numvel^n - \numhvelNL{n}||_3^3}$ with some $c>0$, so the relative dissipation functional provides some additional control over the difference in velocity.

 Estimate \eqref{eqn:lemma_rel_diss} again coincides with the continuous estimate, up to residual terms in the last line of equation \eqref{eqn:lemma_rel_diss} and the additional terms proportional to $\alpha$. We treat the residual terms separately in two further lemmas.
 \begin{remark}
 	\label{remark:coarse_data}
 	Let us assume that the measured data is not available at arbitrarily fine resolution, but instead, it is only available on a coarse grid with spacing $H$. Using Lemma \ref{lemma:dt_proj_estimates} we have $||\sfrac{\hflow^n_H}{\hdens^n_H} - \hvel^n||_2 \leq C H$, with $\hflow^n_H \coloneqq I_H(\hflow^n)$ and $\hdens^n_H$ defined analogously. Then, in \eqref{eqn:lemma_rel_diss}, the term $\frac{\mu^2}{4\alpha}C \xStep^2$ needs to be replaced by $\frac{\mu^2}{4\alpha}C H^2$. This term is then treated the same way as the contribution due to the measurement error $\error$ for the final error bound, Theorem \ref{thm:unif_conv}, and will not decay even as $\xStep, \tStep \rightarrow 0$.
 	
 	This reflects the observations from \cite{lissy_perrin_2025} and \cite{titi_collins_2025} about high-frequency modes which are not resolved by the measurement data potentially leading to non-convergence of the observer system. If the original system state has high-frequency modes with large amplitude, the linear interpolant of the measured data will not approximate $\hvel$ well, leading to a large constant in the above estimate and thus overall poor performance of the observer.
 \end{remark}
\begin{proof}[Proof of Lemma \ref{lemma:rel_diss}]
	Recall the discrete mass balance of the observer
	\begin{equation}
		\label{eqn:pw_const_ident}
	\ip{\ddt\numdens^n}{q} + \ip{\partial_x \numflow}{q} = \ip{\Pi_\xStep \src_1^n}{q}, \quad \forall q \in Q_\xStep.
	\end{equation}
	Testing with the indicator function on each cell of the mesh, because $\ddt\numdens^n$,  $\partial_x \numflow^n$ and $\Pi_\xStep \src_1^n$ are piecewise constant, we can identify $\ddt\numdens^n = -\partial_x \numflow^n + \Pi_\xStep \src_1^n$.
	 Thus, despite $(\numEnthalp^n - \numhEnthalp^n) \notin Q_\xStep$, we have
	 \begin{equation}
	 		\begin{aligned}
	 		\ip{\hamil'(\numu^n) - \hamil'(\numhu^{n})}{\ddt(\numu^n - \numhu^n)} &=&& \ip{\ddt(\numdens^n - \numhdens^n)}{\numEnthalp^n - \numhEnthalp^n} + \ip{\ddt(\numvel^n - \numhvelNL{n})}{\numflow^n - \numhflow^n} \\
	 		&=&& -\ip{\partial_x(\numflow^n - \numhflow^n)}{\numEnthalp^n - \numhEnthalp^n} - \ip{\res_1^n}{\numEnthalp^n - \numhEnthalp^n} \\
	 		&&& +\ip{\partial_x(\numflow^n - \numhflow^n)}{\numEnthalp^n - \numhEnthalp^n} - \ip{\res_2^n}{\numflow^n - \numhflow^n}\\
	 		&&& - \gamma\ip{|\numvel^n|\numvel^n - |\numhvelNL{n}|\numhvelNL{n}}{\numflow^n - \numhflow^n}  \\
	 		&&& - \mu\ip{\numvel^n -\numhvel^n }{\numflow^n - \numhflow^n}- (\Enthalp_\partial^n- \hEnthalp_{\partial}^n)(\numflow^n - \numhflow^n)\biggl\lvert_{x=0}^{x=\ell}\\
	 		&&&+\mu \ip{\numerror^n }{\numflow^n - \numhflow^n}
	 	\end{aligned}
	 \end{equation}
	At $x=0$, we have $\numflow^n =\numhflow^n$ and at $x=\ell$ we have $\Enthalp_\partial^n = \hEnthalp_{\partial}^n$, so the boundary contributions vanish.
	We start by estimating the nudging term. Because $\numhvel^n \neq \numhvelNL{n}$, we rewrite the nudging term as
		\begin{equation}
		\begin{aligned}
			-\mu\ip{\numvel^n -\numhvel^n }{\numflow^n - \numhflow^n} &=	-\mu\ip{\numvel^n -\numhvelNL{n} }{\numflow^n - \numhflow^n} 	-\mu\ip{\numhvelNL{n} -\numhvel^n }{\numflow^n - \numhflow^n} \\
			&\leq 	-\mu\ip{\numvel^n -\numhvelNL{n} }{\numflow^n - \numhflow^n} 	+\frac{\mu^2}{4\alpha}||\numhvelNL{n} -\numhvel^n||_2^2 + \alpha||\numflow^n - \numhflow^n||_2^2,
		\end{aligned}
	\end{equation}
	for any $\alpha \in \symbb{R}^+$.
	Applying the projection error estimates from Lemma \ref{lemma:dt_proj_estimates}, we get
	\begin{equation}
		\frac{\mu^2}{4\alpha}||\numhvelNL{n} -\numhvel^n||_2^2\leq \frac{\mu^2}{4\alpha}C \xStep^2,
		\label{eqn:disvel_proj_err}
	\end{equation}
	where $C$ depends only on the bounds from assumption \ref{ass:unif}.
	 Writing 
	 \begin{equation}
	 	\label{eqn:hflow_split}
	 	\numflow^n - \numhflow^n = \numhdens^n(\numvel^n-\numhvelNL{n}) + \numvel^n(\numdens^n -\numhdens^n)
	 \end{equation}  and  using the uniform bounds from \ref{ass:unif} and \ref{ass:discrete} together with  Young's inequality for $\alpha \in \symbb{R}^+$, we get 
	\begin{equation}
		\begin{aligned}
			-\mu\ip{\numvel^n -\numhvelNL{n} }{\numflow^n - \numhflow^n} &\leq -\mu\underbar\dens||\numvel^n-\numhvelNL{n}||_2^2 - \mu\ip{\numvel^n-\numhvelNL{n}}{\numvel^n(\numdens^n -\numhdens^n)}\\
			&\leq - \frac 12\mu\underbar\dens||\numvel^n-\numhvelNL{n}||_2^2 + \frac 12 \mu \frac{\overbar{v}^2}{\underbar{\dens}} ||\numdens^n - \numhdens^n||_2^2.
		\end{aligned}
	\end{equation}
 Also, due to equation \eqref{eqn:hflow_split}, with $C = \max(\overbar\dens,\overbar\vel)$, we get
 \begin{equation}
 	\alpha||\numflow^n - \numhflow^n||_2^2 \leq  \alpha C^2 ||\numu^n -\numhu^n||_2^2,
 \end{equation}
  	To estimate the measurement error term, we use Young's inequality with $\alpha \in \symbb{R}^+$ and the above identity and obtain
 	\begin{equation}
 		\mu \ip{\numerror^n }{\numflow^n - \numhflow^n} \leq \frac{\mu^2C^2}{\alpha}||\numerror^n||_2^2 + \alpha ||\numu^n - \numhu^n||_2^2 \leq \frac{\mu^2C^2}{\alpha}||\error||^2_{\besov{\infty}{2}} + \alpha ||\numu^n - \numhu^n||_2^2. 
 	\end{equation}
 	Lastly, we estimate the relative dissipation term. For any $\vel,\hvel \in \symbb{R}$, it holds that 
 	\begin{align}
 		\label{eqn:fric_est}
 	|\vel|\vel - |\hvel|\hvel = 2(\vel - \hvel)\int_0^1 |\hvel + s(\vel - \hvel)|\dd{s},
 \end{align}
 	and that the integral expression can be bounded by 
 	\begin{equation}
 		\frac 14 (|\vel| + |\hvel|) \leq \int_0^1 |\hvel + s(\vel - \hvel)|\dd{s} \leq \frac 12 (|\vel| + |\hvel|).
 	\end{equation}
 	Both of these identities can be checked by elementary computations.
 	Thus, using equation \eqref{eqn:hflow_split}, we get that 
 	\begin{equation}
 		\begin{aligned}
 			- \gamma\ip{|\numvel^n|\numvel^n - |\numhvelNL{n}|\numhvelNL{n}}{\numflow^n - \numhflow^n} &\leq \gamma\ip{(|\numvel^n| + |\numhvelNL{n}|)|\numvel^n - \numhvelNL{n}|}{|\numvel^n|(\numdens^n - \numhdens^n )}\\
 			&\quad-\frac 12\gamma\ip{(\numvel^n - \numhvelNL{n})(|\numvel^n| + |\numhvelNL{n}|)}{\numhdens^n(\numvel^n - \numhvelNL{n})} \\
 			&\leq 2\gamma \frac{\overbar{v}^3}{\underline{\rho}}||\numdens^n - \numhdens^n||^2_2 -\symcal{D}(\numu^n,\numhu^n).
 		\end{aligned} 
 	\end{equation}
 	The last estimate above relies on \ref{ass:unif}, \ref{ass:discrete} and weighted Young's inequality. 
 	

\end{proof}

\begin{lemma}[Estimate for $\res_1^n$]
	\label{lemma:res1}
	Let Assumption \ref{ass:unif} hold and $\numhu$ be the given by  \eqref{def:proj_state}. Then, it holds that
	\begin{equation}
		\label{eqn:res1_norm}
		||\res_1^n||_2 \leq \frac{\tau}{2}||\partial_{tt}\hdens||_{\besov{\infty}{2}}.
	\end{equation}
	Additionally, let Assumptions \ref{ass:discrete} and \ref{ass:convexity} hold. Then, there exists a $C>0$, such that for any $\alpha \in \symbb{R}^+$, the following estimate holds:
	\begin{equation}
		-\ip{\mathrm{res}_1^n}{\numEnthalp^n - \numhEnthalp^n} \leq \alpha C ||\numu^n - \numhu^n||_2^2 + \frac{\tau^2}{16\alpha}||\partial_{tt}\hdens||_{\besov{\infty}{2}}^2,
	\end{equation}
	where  $C \coloneqq 2\max\{\overbar{v}^2,\overbar{C}_{\ppot''}^2\}$ is independent of $n$, $\tStep$ and $\xStep$.
\end{lemma}
\begin{proof}
	By the choice of projection operators and the commuting diagram property,\\ ${\Pi_h(\partial_t\hdens^n) = -\Pi_\xStep(\partial_x\hflow^n) + \Pi_\xStep(\src_1^n) = -\partial_x I_h(\hflow^n)}+ \Pi_\xStep(\src_1^n)$, we have 
	$$
	-\ip{\mathrm{res}_1^n}{q} = \ip{\partial_t \numhdens^n - \ddt\numhdens^n }{q},\quad \forall q\in L^2(\xDom).
	$$
	From Taylor's theorem and the continuity of the $L^2$--projection, it holds that
	\begin{equation}
||\res_1^n||_2 = 	||\partial_t \numhdens^n - \ddt\numhdens^n||_2 \leq \frac{\tau}{2}||\partial_{tt}\numhdens||_{\besov{\infty}{2}} = \frac{\tau}{2}||\partial_{tt}\hdens||_{\besov{\infty}{2}}.	
	\end{equation}
	By weighted Young's inequality, with $\alpha \in \symbb{R}^+$,  we get 
	$$
	\ip{\res_1^n}{\numEnthalp^n - \numhEnthalp^n}\leq \alpha||\numEnthalp^n - \numhEnthalp^n||_2^2 + \frac{1}{4\alpha}||\res_1^n||_2^2.
	$$
	The first term can be estimated using 
	\begin{equation}
		|\numEnthalp^n - \numhEnthalp^n| = | \frac12 (|\numvel^n|^2 - |\numhvelNL{n}|^2) + P'(\numdens^n) - P'(\numhdens^n)| \leq \overbar{v}|\numvel^n - \numhvelNL{n}| + \overbar{C}_{P''}|\numdens^n - \numhdens^n|,
		\label{eqn:proj_err_enthalp}
	\end{equation}
	such that $||\numEnthalp^n - \numhEnthalp^n||_2^2  \leq C||\numu^n - \numhu^n||_2^2$ for some $C>0$ depending only on the uniform bounds from \ref{ass:unif} and \ref{ass:discrete}.
\end{proof}
\begin{lemma}[Estimate for $\res_2^n$]
	\label{lemma:res2h}
Let the assumptions of Lemma \ref{lemma:ddt_hamil} hold. Further, let \ref{ass:meas_err} hold.
	 Then, for all $0 < \tStep \leq \tStep_0$ and $0 < \xStep \leq \xStep_0$, there exists a constant $C>0$, such that for any $\alpha \in \symbb{R}^+$ it holds that
	\begin{equation}
	\begin{aligned}
		- \ip{\res_2^n}{\numflow^n - \numhflow^n} \leq& \alpha (||\numu^n - \numhu^n||_2^2 + ||\numu^{n-1} - \numhu^{n-1}||_2^2) + C(1 + \frac{1}{\alpha})(\xStep^2 + \tStep^2)  \\
		&+ \frac12 \diss(\numu^n,\numhu^n) +\ddt(\ip{\hEnthalp^n - \numhEnthalp^n}{\numdens^n - \numhdens^n}).
	\end{aligned}
	\end{equation}
	The constant $C$ depends on $\tStep_0$, $\xStep_0$ and the properties of the original system solution $\hu$.
\end{lemma}
\begin{proof}
	We begin by recalling that $\hu$ satisfies the weak formulation of the barotropic Euler equations. In particular, for any $r \in H^1(\xDom)$ it holds that
	$$
	0 = \ip{\partial_t\hvel^n}{r} - \ip{\hEnthalp^n}{\partial_x r} + \gamma \ip{\abs{\hvel^n}\hvel^n}{r} + \hEnthalp_\partial^n r\biggl\lvert_{x=0}^{x=\ell} - \ip{\src_2^n}{r}.
	$$
	Adding the above weak formulation tested with $ r = \numflow^n - \numhflow^n$ to the residual term we get
	\begin{equation}
	\begin{aligned}
		- \ip{\res_2^n}{\numflow^n - \numhflow^n} =& \ip{\partial_t\hvel^n - \ddt(\numhvelNL{n})}{\numflow^n - \numhflow^n} - \ip{\hEnthalp^n - \numhEnthalp^n}{\partial_x (\numflow^n - \numhflow^n)} 
		\\\quad&+ \gamma \ip{\abs{\hvel^n}\hvel^n-\abs{\numhvelNL{n}}\numhvelNL{n}}{\numflow^n - \numhflow^n} +\ip{\src_2^n - \src_2^n}{\numflow^n - \numhflow^n}
		\\&+ (\hEnthalp_\partial^n - \hEnthalp_\partial^n) (\numflow^n - \numhflow^n)\biggl\lvert_{x=0}^{x=\ell} \label{eqn:proj_diss}
		\eqqcolon (A) + (B) + (C) + 0 + 0
	\end{aligned}
	\end{equation}
	The boundary terms vanish due to $\numflow^n = \numhflow^n$ at $x = 0$ and because $\numu$ satisfies the boundary condition on $\hEnthalp_\partial^n$  weakly at $x=\ell$.
	
	We estimate $(A)$ by Young's inequality. For any $\alpha \in \symbb{R}^+$, the following estimate holds:
	\begin{align}
		\ip{\partial_t\hvel^n - \ddt(\numhvelNL{n})}{\numflow^n - \numhflow^n} \leq \frac{1}{4\alpha}||\partial_t\hvel^n - \ddt(\numhvelNL{n})||_2^2 + \alpha||\numflow^n - \numhflow^n||_2^2.
	\end{align}
	Using equation \eqref{eqn:hflow_split} and the uniform bounds from \ref{ass:unif} and \ref{ass:discrete}, we have 
	$$
	\alpha||\numflow^n - \numhflow^n||_2^2 \leq\alpha C ||\numu^n -\numhu^n||_2^2,
	$$
	with $C$ depending only on the uniform bounds \ref{ass:unif}.
	Further, by Lemma \ref{lemma:dt_proj_estimates}, we have 
	\begin{equation}
		||\partial_t\numvel^n - \ddt(\numhvelNL{n})||_2^2 \leq C \left(\xStep^2 + \tStep^2\right),
	\end{equation}
	for some $C>0$ independent of $n$, $\tStep$ and $\xStep$.
	
	For $(B)$ in \eqref{eqn:proj_diss}, we apply the discrete formulation of the mass balance, getting 
	\begin{align}
		\ip{\hEnthalp^n - \numhEnthalp^n}{\partial_x (\numflow^n - \numhflow^n)} = -\ip{\hEnthalp^n - \numhEnthalp^n}{\ddt(\numdens^n - \numhdens^n)} - \ip{\res_1^n}{\hEnthalp^n - \numhEnthalp^n} \label{eqn:enthalperr}
	\end{align}
	Using Young's inequality gives
	\begin{align}
		\ip{\res_1^n}{\hEnthalp^n - \numhEnthalp^n} \leq \frac 12 ||\res_1^n||_2^2 + \frac 12 ||\hEnthalp^n - \numhEnthalp^n||_2^2 \leq C^2 \tau^2+ \frac 12 ||\hEnthalp^n - \numhEnthalp^n||_2^2.
	\end{align}
	by Lemma \ref{lemma:res1}.
	Combining estimate \eqref{eqn:proj_err_enthalp} with the projection error estimates and Lemma \ref{lemma:dt_proj_estimates}, we get
	\begin{equation}
		||\hEnthalp^n - \numhEnthalp^n||_2  \leq \overbar{v}||\hvel^n - \numhvelNL{n}||_2 + \overbar{C}_{P''}||\hdens^n - \numhdens^n||_2\leq C \xStep,
		\label{eqn:enthalp_proj_est}
	\end{equation}
	for some constant $C>0$.
	We apply discrete integration by parts to the first term in \eqref{eqn:enthalperr}, $(\ddt a^n)b^n = -a^{n-1}\ddt b^n + \ddt(a^nb^n)$.	
	Additionally, applying Young's inequality gives
	\begin{equation}
		\begin{aligned}
		-\ip{\hEnthalp^n - \numhEnthalp^n}{\ddt(\numdens^n - \numhdens^n)} &= \ip{\ddt(\hEnthalp^n - \numhEnthalp^n)}{\numdens^{n-1} - \numhdens^{n-1}}  -\ddt\left(\ip{\hEnthalp^n - \numhEnthalp^n}{\numdens^n - \numhdens^n}\right)\\
		&\leq \alpha ||\numdens^{n-1} - \numhdens^{n-1}||_2^2 + \frac{1}{4\alpha}||\ddt(\hEnthalp^n - \numhEnthalp^n)||_2^2 -\ddt(\ip{\hEnthalp^n - \numhEnthalp^n}{\numdens^n - \numhdens^n}).
		\end{aligned}
	\end{equation}
	 The term $\ddt(\ip{\hEnthalp^n - \numhEnthalp^n}{\numdens^n - \numhdens^n})$ will yield a telescoping sum when adding up all the time steps for Gronwall's Lemma in the final convergence proof.
	We split the projection term as follows: 
	$$
	||\ddt(\hEnthalp^n - \numhEnthalp^n)||_2^2 \leq 3\left(||\ddt\hEnthalp^n - \partial_t\hEnthalp^n||_2^2 + ||\partial_t(\hEnthalp^n - \numhEnthalp^n)||_2^2+ ||\partial_t\numhEnthalp^n - \ddt\numhEnthalp^n||_2^2\right).
	$$
	With the bounds and regularity from \ref{ass:unif}, Taylor's theorem and the projection error estimates, the separate terms can be bounded as 
	\begin{align}
		||\ddt\hEnthalp^n - \partial_t\hEnthalp^n||_2^2 &\leq C\tStep^2||\partial_{tt} \hEnthalp||^2_{\besov{\infty}{2}}\leq \tilde C \tStep^2,\\
		||\partial_t(\hEnthalp^n - \numhEnthalp^n)||_2^2 &\leq C \xStep^2||\partial_t \hEnthalp||^2_{L^{\infty}(0,\infty;H^1(\xDom))} \leq \tilde C \xStep^2, \\
		||\partial_t\numhEnthalp^n - \ddt\numhEnthalp^n||_2^2&\leq C\tStep^2||\partial_{tt} \numhEnthalp||^2_{\besov{\infty}{2}}\leq \tilde C \tStep^2.
	\end{align}
	The constants depend only on the bounds from Assumption \ref{ass:unif} and norms of the original system state $\hu$.
	
	This leaves estimating $(C)$ from \eqref{eqn:proj_diss}. We will see that parts of this term can be absorbed into the relative dissipation functional from Lemma \ref{lemma:rel_diss}.
	We rewrite $(C)$ in the following fashion and estimate the contributions individually: 
	\begin{equation}
	\begin{aligned}
		\gamma\ip{ \,\abs{\hvel^n}\hvel^n - \abs{\numhvelNL{n}}\numhvelNL{n}}{\numflow^n - \numhflow^n} =& \gamma\ip{ (|\hvel^n| - |\numhvelNL{n}|)(\hvel^n - \numhvelNL{n})}{\numflow^n - \numhflow^n} \\
		&+\gamma\ip{ |\numhvelNL{n}|(\hvel^n - \numhvelNL{n})}{\numflow^n - \numhflow^n}  \\
		&+ \gamma\ip{\numhvelNL{n}(|\hvel^n| - |\numhvelNL{n}|)}{\numflow^n - \numhflow^n} \eqqcolon (D) + (E) + (F).
	\end{aligned}
	\end{equation}
	Using Hölder's and Young's inequality, for any $\alpha_1, \alpha_2 \in \symbb{R}^+$, by using equation \eqref{eqn:hflow_split} we have 
	\begin{equation}
		\begin{aligned}
			(D) &\leq \overbar{\rho}\gamma ||(\hvel^n - \numhvelNL{n})^2||_{3/2}||\numvel^n - \numhvelNL{n}||_3 + \overbar{v}\gamma||(\hvel^n - \numhvelNL{n})^2||_2||\numdens^n -\numhdens^n||_2\\
			&\leq \frac{2}{3\alpha_1^{3/2}}||\hvel^n - \numhvelNL{n}||_3^3 +\frac{1}{3}\alpha_1^3||\numvel^n - \numhvelNL{n}||_3^3 + \frac{1}{2\alpha_2}||(\hvel^n - \numhvelNL{n})||_4^4 + \frac{\alpha_2}{2}||\numdens^n -\numhdens^n||_2^2.
		\end{aligned}
	\end{equation}
	 Choosing $\alpha_1$ sufficiently small, we have $\frac{1}{3}\alpha_1^3||\numvel^n - \numhvelNL{n}||_3^3 \leq \frac 14 \diss(\numu^n,\numhu^n)$. Further, from the projection error estimates, we have $||\hvel^n - \numhvelNL{n}||_3^3 \leq C \xStep^3$ and ${||(\hvel^n - \numhvelNL{n})||_4^4 \leq C \xStep^4}$. For $\xStep$ sufficiently small, the $\xStep^3$ and $\xStep^4$ terms are both bounded by $\xStep^2$.
	
	The remaining two terms are estimated by 
	\begin{equation}
		\begin{aligned}
			(E) + (F) &\leq \gamma^2 \frac{1}{\alpha_3}||\numhvelNL{n}||_\infty||\hvel^n - \numhvelNL{n}||_2^2 + \alpha_3 ||(\numflow^n - \numhflow^n)|\numhvelNL{n}|^{1/2}||_2^2 \\
			&\leq \frac{Ch^2}{\alpha_3} + 2\alpha_3 \overbar{\dens} \int_0^\ell \numhdens^n \abs{\numvel^n - \numhvelNL{n}}^2 |\numhvelNL{n}|\dd{x}+ 2\alpha_3 \overbar{\vel}^2||\numhvelNL{n}||_\infty||\numdens^n  - \numhdens^n||_2^2,\\
			&\leq \frac{Ch^2}{\alpha_3} + \frac{1}{4} \symcal{D}(\numu^n, \numhu^n) + \alpha_3 C ||\numu^n  - \numhu^n||_2^2,
		\end{aligned}
	\end{equation}
	
	for $\alpha_3 \in \symbb{R}^+$ sufficiently small.
	Gathering up the individual estimates, and renaming constants gives the statement of the lemma.
\end{proof}
The previous lemmas thus show, that the discrete time derivative of the relative energy is consistent with the continuous time derivative. We gather the overall estimates in the following result:
\begin{theorem}[Estimate for $\ddt\hamil(\numu^n,\numhu^n)$]
	Let $\numhu^n$ be the projected original system state given by \eqref{def:proj_state} and let $\numu^n$ be a solution of system \eqref{eqn:dis_scheme_1} -- \eqref{eqn:dis_scheme_2}. Let $0 < \tStep \leq \tStep_0$ and $0 < \xStep \leq \xStep_0$. Let Assumptions \ref{ass:unif} --\ref{ass:subsonic}, \ref{ass:discrete} and \ref{ass:meas_err} hold. Then, there exist constants ${C_1,C_2,C_3,C_4>0}$ such that for any $\alpha \in \symbb{R}^+$, the following estimate holds:
	\begin{equation}
		\begin{aligned}
			\ddt\hamil(\numu^n,\numhu^n)	 \leq& - \frac 12\mu\underbar\dens||\numvel^n-\numhvelNL{n}||_2^2 + \frac 12 \mu \frac{\overbar{v}^2}{\underbar{\dens}} ||\numdens^n - \numhdens^n||_2^2+ 2\gamma \frac{\overbar{v}^3}{\underline{\rho}}||\numdens^n - \numhdens^n||_2^2 \\
			 & +\frac 12 \max(1,\overbar{C}_{\ppot'''}) C_t|| \numu^n - \numhu^n||_2^2-\frac 12 \diss(\numu^n,\numhu^n) + \frac{\mu^2C_1}{\alpha}||\error||^2_{\besov{\infty}{2}}
			 \\& - \frac{\tStep}{2\underline C}||\ddt(\numu^n - \numhu^n)||_2^2 + \ddt(\ip{\hEnthalp^n - \numhEnthalp^n}{\numdens^n - \numhdens^n})\\
			&+C_2 (\alpha + \tStep)||\numu^n - \numhu^n||_2^2 + C_3( \alpha + \tStep) ||\numu^{n-1} - \numhu^{n-1}||_2^2 + C_4 (1 + \frac{1}{\alpha})(\xStep^2 + \tStep^2).\\
		\end{aligned}
	\end{equation}
	Further, the constants $C_1,C_2,C_3$ and $C_4$ are independent of $n$, $\tStep$ and  $\xStep$. The relative dissipation functional $\diss$ is given by \eqref{eqn:rel_diss}.
\end{theorem}
\subsection{Modified relative energy.}

In this section, following the strategy from \cite{kunkel_obs}, we give a discrete functional $\disG$, which satisfies the norm equivalence requirement \ref{req:norm_equiv}. To handle the discretization errors and the discrete time derivative, the requirement \ref{req:dissipative} is modified:
\begin{enumerate}[leftmargin = 1.5cm,label=\textbf{(R\arabic*h)},ref=\textbf{(R\arabic*h)},start=2]
	\item \label{req:mod_dissipative} For any choice of $\mu$, there exist sufficiently small choices of $\delta$ and the bounds in Assumption \ref{ass:smallderivs} such that 
	\begin{equation}
		\begin{aligned}
				\ddt \left(\hamil(\numu^n,\numhu^n) + \delta\disG(\numu^n,\numhu^n)\right) \leq& -C_1\left(\hamil(\numu^n,\numhu^n) + \delta\disG(\numu^n,\numhu^n)\right) 
\\&+ C_2(\tStep^2 + \xStep^2) + C_3 ||\error||^2_{\besov{\infty}{2}},
		\end{aligned}	
	\end{equation}
	with constants $C_1,C_2,C_3>0$ depending on the bounds in Assumptions \ref{ass:unif} -- \ref{ass:smallderivs}, \ref{ass:meas_err} and \ref{ass:discrete}. Additionally, we require that  $C_1,C_2$ and $C_3$ are independent of $n$, $\xStep$ and $\tStep$.
\end{enumerate}
Once \ref{req:mod_dissipative} has been established, applying Gronwall's Lemma shows Theorem \ref{thm:unif_conv}.
We define the auxiliary functional in the following way: Let 
\begin{equation}
	\numnM^n(x) \coloneqq -\int_0^x \numdens^n(s)\dd{s},\quad \text{ and } \quad	\numhM^n(x) \coloneqq -\int_0^x \numhdens^n(s)\dd{s}.
\end{equation}
Then, we define $\disG(\numu^n,\numhu^n)$ as 
\begin{equation}
	\disG(\numu^n,\numhu^n) \coloneqq \ip{\numnM^n - \numhM^n}{\numvel^n-\numhvelNL{n}}.
\end{equation}
Note that this definition differs from the definition in \cite{kunkel_obs}, where the functions $\nM$ and $\hM$ integrate $\flow$ and $\hflow$ in time, respectively. The time integration causes problems, because the estimates then depend on $n$, which is exactly what we want to avoid.
We start by verifying the norm equivalence.
\begin{lemma}
	\label{lemma:mod_norm_eqiv}
	 	Let $\numhu^n$ be the projected original system state given by \eqref{def:proj_state} and let $\numu^n$ be a solution of system \eqref{eqn:dis_scheme_1} -- \eqref{eqn:dis_scheme_2}. Let $\cpoin$ be the Poincaré constant of the unit interval. Then, 
	\begin{align}
		\label{eqn:abs_disG}
	|\disG(\numu^n,\numhu^n)| \leq& \frac12 \cpoin \ell \left(||\numvel^n - \numhvelNL{n}||_2^2 +||\numdens^n - \numhdens^n||_2^2\right) 
	\end{align}
	Further, let assumptions \ref{ass:unif} --\ref{ass:subsonic} and \ref{ass:discrete} hold and let $\delta\in\symbb{R}^+$ with  
	$$
	\delta \leq \frac{c_0}{\cpoin \ell}
	$$
	and $c_0,C_0$ from Lemma \ref{lemma:norm_equiv}.
	Then, 
	\begin{equation}
		\frac 12 c_0 ||\numu^n - \numhu^n||_2^2 \leq \hamil(\numu^n,\numhu^n) + \delta \disG(\numu^n,\numhu^n) \leq \frac{3}{2} C_0||\numu^n - \numhu^n||_2^2.
		\label{eqn:mod_rel_energy}
	\end{equation}
\end{lemma}
\begin{proof}
		We start by applying Hölder's inequality and Poincaré's inequality, which is applicable due to $\numnM^n = \numhM^n = 0$ in $x=0$. This gives
		\begin{equation}
			\begin{aligned}
				\disG(\numu^n,\numhu^n) &=  \ip{\numnM^n - \numhM^n}{\numvel^n-\numhvelNL{n}} \leq ||\numnM^n - \numhM^n||_2 ||\numvel^n - \numhvelNL{n}||_2 \\
				& \leq \cpoin \ell || \partial_x(\numnM^n - \numhM^n)||_2||\numvel^n - \numhvelNL{n}||_2.
				\label{eqn:Gest1}
			\end{aligned}
		\end{equation}
	By definition, we have $|| \partial_x(\numnM^n - \numhM^n)||_2 = ||\numdens^n-\numhdens^n||_2$. Finally applying Young's inequality gives equation \eqref{eqn:abs_disG}.
	The norm equivalence is a direct consequence of equation \eqref{eqn:abs_disG} for the given choice of $\delta$.

\end{proof}

Next, we compute the discrete time derivative of $\disG(\numu^n,\numhu^n)$.
\begin{lemma}[Discrete time derivative of $\disG$]
	\label{lemma:ddt_disG}
	It holds that
	\begin{equation}
	\begin{aligned}
		\ddt \disG(\numu^n,\numhu^n) =& \ip{\ddt(\numnM^n - \numhM^n )}{\numvel^{n} - \numhvelNL{n}} + \ip{\numnM^n - \numhM^n  }{\ddt(\numvel^n - \numhvelNL{n})}\\
		&- \tStep\ip{\ddt(\numnM^n - \numhM^n )}{\ddt(\numvel^n - \numhvelNL{n})}\eqqcolon (I) + (II) + (III).
	\end{aligned}
	\end{equation}
\end{lemma}
\begin{proof}
	The claim follows directly by applying the definition of the discrete time derivative.
\end{proof}

We now estimate the contributions individually.
\begin{lemma}[Estimating (I) from Lemma \ref{lemma:ddt_disG}]
	\label{lemma:disG_est_1}
	Let $\numhu^n$ be the projected original system state given by \eqref{def:proj_state} and let $\numu^n$ be a solution of system \eqref{eqn:dis_scheme_1} -- \eqref{eqn:dis_scheme_2}.
	 Let assumptions \ref{ass:unif} and \ref{ass:discrete} hold. Then, there exists a $C>0$ such that, for any $\alpha \in \symbb{R}^+$,	we get 
	\begin{equation} \ip{\ddt(\numnM^n - \numhM^n)}{\numvel^{n} - \numhvelNL{n}} \leq \frac 32\overbar\dens ||\numvel^n-\numhvelNL{n}||_2^2 +  \frac{\overbar{\vel}^2}{2\overbar\dens} ||\numdens^n - \numhdens^n||_2^2 +\frac{C^2\ell^2 \tau^2}{4\alpha} + \alpha||\numvel^n - \numhvelNL{n}||_2^2.
	\end{equation}	
	The constant $C$ is independent of $n$, $\tStep$ and $\xStep$.
\end{lemma}
\begin{proof}
We begin by computing $\ddt(\numnM^n -\numhM^n)$. Let $\chi_{[0,x]}$ denote the indicator function on $[0,x]$.
By definition, 
\begin{equation}
\begin{aligned}
	\ddt(\numnM^n -\numhM^n)(x) &= -\ddt\left( \ip{\numdens^n -\numhdens^n}{\chi_{[0,x]}}\right) = -\ip{\ddt(\numdens^n -\numhdens^n)}{\chi_{[0,x]}} \\
	&= \ip{\partial_x(\numflow^n -\numhflow^n)}{\chi_{[0,x]}} - \ip{\Pi_\xStep \src_1^n - \Pi_\xStep \src_1^n}{\chi_{[0,x]}} + \ip{\res_1^n}{\chi_{[0,x]}} \\
	&=  (\numflow^n - \numhflow^n)(x) - (\numflow^n - \numhflow^n)(0) + \ip{\res_1^n}{\chi_{[0,x]}} \label{eqn:dt_Mdif},
\end{aligned}
\end{equation}

with $(\numflow^n - \numhflow^n)(0) = 0$ due to the boundary condition at $x = 0$.
Using equation \eqref{eqn:hflow_split}, the uniform bounds from \ref{ass:unif} and \ref{ass:discrete}, and Young's inequality gives
\begin{equation}
	\begin{aligned}
			(I) &= \ip{\numhdens^n(\numvel^n - \numhvelNL{n} )}{\numvel^{n} - \numhvelNL{n}} + \ip{ \numvel^n (\numdens^n - \numhdens^n)}{\numvel^{n} - \numhvelNL{n}} + \ip{ \int_0^x\res_1^n(s)\dd{s}}{\numvel^{n} - \numhvelNL{n}}  \\
			&=\frac 32\overbar\dens ||\numvel^n-\numhvelNL{n}||_2^2 +  \frac{\overbar{\vel}^2}{2\overbar\dens} ||\numdens^n - \numhdens^n||_2^2+\ip{ \int_0^x\res_1^n(s)\dd{s}}{\numvel^{n} - \numhvelNL{n}} . 
		\end{aligned}
\end{equation}
	Using Hölder's and Young's inequality with $\alpha \in \symbb{R}^+$ on the last term gives 
	\begin{equation}
		\ip{ \int_0^x\res_1^n(s)\dd{s}}{\numvel^{n} - \numhvelNL{n}} \leq ||\int_0^x\res_1^n(s)\dd{s}||_\infty \ell^{1/2}||\numvel^{n} - \numhvelNL{n}||_2\leq \frac{C^2\ell^2 \tau^2}{4\alpha}+ \alpha||\numvel^n - \numhvelNL{n}||_2^2.
		\label{eqn:res1_disG_est}
	\end{equation}
	The last inequality uses  that 
	\begin{equation}
		||\int_0^x\res_1^n(s)\dd{s}||_\infty \leq ||\res_1^n||_1 \leq \ell^{1/2} ||\res_1^n||_2 \leq C\ell^{1/2} \tStep,
	\end{equation}
	with $C>0$ depending only on the bounds in \ref{ass:unif} and norms of $\hu$, due to estimate \eqref{eqn:res1_norm} from Lemma \ref{lemma:res1}.
	Gathering the previous estimates gives the statement of the lemma.
\end{proof}
\begin{lemma}[Estimate for (III) from Lemma \ref{lemma:ddt_disG}]
	 Let $\delta\in \symbb{R}^+$ be the constant from the modified relative energy \eqref{eqn:mod_rel_energy}.	 Under the assumptions of Lemma \ref{lemma:disG_est_1}, there exist constants $C_1,C_2>0$ such that
	\begin{align}
		- \tStep\ip{\ddt(\numnM^n - \numhM^n )}{\ddt(\numvel^n - \numhvelNL{n})} \leq \frac {\tStep}{2\delta} \underline{C} ||\ddt(\numu^n - \numhu^n)||_2^2 + \delta\tau C_1 ||\numu^n - \numhu^n||_2^2 +\delta C_2{\ell^2 \tStep^3}, \label{eqn:numdiss2}
	\end{align}
	where $\underline{C}$ is the constant from Lemma \ref{lemma:bounded} and $C_1$ and $C_2$ depend  on $\underline{C}$ and the bounds from \ref{ass:unif} and \ref{ass:discrete}, but are independent of $n$, $\tStep$ and $\xStep$.
\end{lemma}
\begin{proof}
	Using 
	$$
	\ddt(\numnM^n -\numhM^n)(x) = (\numflow^n - \numhflow^n)(x)+ \int_0^x\res_1^n(s)\dd{s}\
	$$
	from \eqref{eqn:dt_Mdif}, we get 
	\begin{equation}
		\begin{aligned}
		-\tStep \ip{\ddt(\numnM^n - \numhM^n )}{\ddt(\numvel^n - \numhvelNL{n})} = &-\tau\ip{\numflow^n -\numhflow^n}{\ddt(\numvel^n - \numhvelNL{n})} \\&-\tau \ip{ \int_0^x\res_1^n(s)\dd{s}}{\ddt(\numvel^n - \numhvelNL{n})}.
		\label{eqn:ddtM_numDiss}
		\end{aligned}
	\end{equation}
	 By Hölder's  and Young's inequality, for any  $\alpha \in \symbb{R}^+$, we get 
	\begin{equation}
		\begin{aligned}
			-\tStep \ip{\numflow^n - \numhflow^n}{\ddt(\numvel^n - \numhvelNL{n})} &\leq \tStep||\numflow^n - \numhflow^n||_2 ||\ddt(\numvel^n - \numhvelNL{n})||_2 \\
			&\leq \delta\frac{\tStep}{4\alpha} ||\numflow^n - \numhflow^n||_2^2 + \frac{\alpha\tStep}{\delta}||\ddt(\numvel^n - \numhvelNL{n})||_2^2,
		\end{aligned}
	\end{equation}
	with $\delta$ the constant from the modified relative energy \eqref{eqn:mod_rel_energy}. Using \eqref{eqn:hflow_split}, we have\\ $||\numflow^n - \numhflow^n||_2^2 \leq C||\numu^n -\numhu^n||_2^2$.
	The second term in \eqref{eqn:ddtM_numDiss} is estimated  using Young's inequality, giving
	\begin{align}
		-\tStep \ip{ \int_0^x\res_1^n(s)\dd{s}}{\ddt(\numvel^n - \numhvel^n)} \leq \tStep\delta\frac{C^2\ell^2 \tStep^2}{4\alpha} + \frac{\alpha \tStep}{\delta}||\ddt(\numvel^n - \numhvelNL{n})||_2^2,
	\end{align}
	similar to equation \eqref{eqn:res1_disG_est} from the proof of Lemma \ref{lemma:ddt_disG}.
	
	Recalling Lemma \ref{lemma:dthp2}, the estimates of the time derivative of $\hamil(\numu^n ,\numhu^n)$ give numerical dissipation of $-\tStep\frac 12\underline{C}||\ddt(\numu^n-\numhu^n)||_2^2$ with $\underline{C}$ from Lemma \ref{lemma:bounded} depending on the bounds from \ref{ass:unif} and \ref{ass:convexity}. Estimating $||\ddt(\numvel^n -\numhvelNL{n})||_2 \leq ||\ddt(\numu^n - \numhu^n) ||_2$, we choose $\alpha = \frac 14 \underline{C}$ such that the $\frac{\alpha \tStep}{\delta}||\ddt(\numvel^n - \numhvelNL{n})||_2^2$-- terms can be absorbed by the remaining numerical dissipation, when adding up all the estimates for the modified relative energy $\hamil(\numu^n ,\numhu^n) + \delta \disG(\numu^n ,\numhu^n)$.
\end{proof}
\begin{lemma}[Estimate for (II) from Lemma \ref{lemma:ddt_disG}]
	\label{lemma:ddt_disG_II}  Let the assumptions from Lemma \ref{lemma:disG_est_1} hold. Further, let Assumption \ref{ass:meas_err} hold. Then, there exists a constant $C>0$ independent $n$, $\tStep$ and $\xStep$, such that for any $\alpha \in \symbb{R}^+$, it holds that 
\begin{equation}
	\begin{aligned}\ip{M^n - \hM^n }{\ddt(\numvel^n - \numhvelNL{n})}\leq& \left(-\frac 12\underline{C}_{\ppot''}+ \frac{\overbar{\vel}^2}{2\underline{\dens}} + \alpha\right)||\numdens^n -\numhdens^n||_2^2 
		\\&+\left(\frac 12 \underline{\dens} +  \frac{\cpoin^2 \ell^2}{\underline{C}_{P''}}(\mu^2 + \gamma^24\overbar{\vel}^2) \right)||\numvel^n - \numhvelNL{n}||_2^2\\
		&+ \frac{C^2 \xStep^2}{4\alpha}+\frac{\mu^2 C^2}{\alpha}||\error||_{\besov{\infty}{2}}^2 - \ip{\res_2^n}{\numnM^n - \numhM^n}.
	\end{aligned}
\end{equation}
\end{lemma}
\begin{proof}
	 Note that by definition, $\numnM^n -\numhM^n\in R_\xStep $, so we can use it as a test function for the discrete formulation \eqref{eqn:dis_scheme_2}. Using the definition of the residual $\res_2^n$,  we get
\begin{equation}
	\label{eqn:enthalp_boundary}
	\begin{aligned}
		(II) &= \ip{\partial_x(\numnM^n - \numhM^n)}{\numEnthalp^n - \numhEnthalp^n} - (\Enthalp_\partial^n - \hEnthalp_\partial^n)(\numnM^n - \numhM^n)\biggl\lvert_{x=0}^{x=\ell}- \mu\ip{\numvel^n - \numhvel^n}{\numnM^n - \numhM^n}\\
		& \quad +\mu\ip{\numerror^n}{\numnM^n - \numhM^n}- \gamma\ip{|\numvel^n|\numvel^n - |\numhvelNL{n}|\numhvelNL{n}}{\numnM^n - \numhM^n} - \ip{\res_2^n}{\numnM^n - \numhM^n} \\ 
		&\eqqcolon (A) + 0 + (B) + (C) + (D) + (E).
	\end{aligned}
\end{equation}
	At $x = \ell$ we have $\hEnthalp_\partial^n = \Enthalp_\partial^n$ as a boundary condition. Further, at $x = 0$, we have ${\numnM^n = \numhM^n  = 0}$. Thus, the boundary contribution vanishes.
	 The residual term  $(E)$ already appears in the statement of the lemma, and will be estimated separately later.
	
	Now, we consider the remaining terms individually.	
	 For the first term, using the definition of $\numnM^n$ and $\numhM^n$, we get 
	\begin{align}
		(A) = -\ip{\numdens^n - \numhdens^n}{\numEnthalp^n - \numhEnthalp^n}.
	\end{align}
Inserting the definitions of $\numEnthalp$ and $\numhEnthalp$, expanding $\ppot'$ using Taylor's theorem and the convexity of the pressure potential $\ppot$ gives 
\begin{equation}
	\begin{aligned}
		-\ip{\numdens^n - \numhdens^n}{\numEnthalp^n - \numhEnthalp^n} &= -\ip{\numdens^n - \numhdens^n}{\frac12((\numvel^n)^2 - (\numhvelNL{n})^2))+ \ppot'(\numdens^n) - \ppot'(\numhdens^n)} \\
		&\leq -\underline{C}_{\ppot''}||\numdens^n -\numhdens^n||_2^2 - \frac 12\ip{\numdens^n - \numhdens^n}{(\numvel^n - \numhvelNL{n} )(\numvel^n + \numhvelNL{n} )} \\
		&\leq \left(-\underline{C}_{\ppot''}+\frac{\overbar{\vel}^2}{2\underline{\dens}}\right)||\numdens^n -\numhdens^n||_2^2 + \frac 12 \underline{\dens}||\numvel^n - \numhvelNL{n}||_2^2.
	\end{aligned}
\end{equation}
	For the last inequality, we use Lemma \ref{lemma:dt_proj_estimates} and Remark \ref{remark:dt_proj_estimates}.
	Next, we estimate the nudging term $(B)$. Recall that $|| \numhvelNL{n}-\numhvel^n||_2 \leq C \xStep$ due to Lemma \ref{lemma:dt_proj_estimates}.
 Let $\alpha \in \symbb{R}^+$. Using the Poincaré's and Young's inequality, we get 
 \begin{equation}
 	\begin{aligned}
 		(B) &\leq \mu \cpoin \ell ||\partial_x(\numnM^n - \numhM^n)||_2\left(||\numvel^n - \numhvelNL{n}||_2 + || \numhvelNL{n}-\numhvel^n||_2\right)\\
 		&\leq \left(\frac{1}{4}\underline{C}_{P''}+\alpha\right)||\numdens^n -\numhdens^n||_2^2 + \mu^2 \cpoin^2 \ell^2 \frac{1}{\underline{C}_{P''}}||\numvel^n - \numhvelNL{n}||_2^2 + \frac{C^2 \cpoin^2 \ell^2\mu^2}{\alpha}\xStep^2.
 	\end{aligned}
 \end{equation}
	
	Then we estimate the error term $(C)$. Again, applying Poincaré's, Hölder's and Young's in\-equality gives
	\begin{equation}
		\mu\ip{\numerror^n}{\numnM^n - \numhM^n} \leq \mu\cpoin \ell||\numerror^n||_2||\numdens - \numhdens||_2 \leq \frac{\mu^2 C^2}{\alpha}||\error||_{\besov{\infty}{2}}^2 + \alpha ||\numdens - \numhdens||_2^2.
	\end{equation}
	
	Lastly, we estimate the friction terms.
	Applying Cauchy-Schwarz, Poincaré's inequality and estimate \eqref{eqn:fric_est} we have 
\begin{equation}
	\begin{aligned}
		(D) &\leq \cpoin \ell \gamma 2\overbar{\vel} ||\partial_x(\numnM^n - \numhM^n )||_2 ||\numvel^n - \numhvelNL{n}||_2 \\
		&\leq \frac{1}{4}\underline{C}_{\ppot''}||\numdens^n - \numhdens^n||_2^2 + \gamma^2 \cpoin^2 \ell^2 \frac{4\overbar{\vel}^2}{\underline{C}_{\ppot''} }||\numvel^n - \numhvelNL{n}||_2^2.
	\end{aligned}
\end{equation}
	\end{proof}
	
	Finally, we need to estimate the remaining PDE residual term. This will be done in the next lemma: 
	\begin{lemma}  Let $\numhu^n$ be the projected original system state given by \eqref{def:proj_state}. Further, let assumptions \ref{ass:unif}, \ref{ass:discrete} and \ref{ass:meas_err} hold. Then, there exists a constant $C>0$, such that for any  $\alpha \in \symbb{R}^+$, it holds that
		\begin{align}
			- \ip{\res_2^n}{\numnM^n - \numhM^n} \leq \alpha\left(1 + \cpoin \ell (2 + \gamma)\right) ||\numdens^n - \numhdens^n||_2^2 + \frac{C}{\alpha}(\tStep^2 + \xStep^2).
		\end{align}
		The constant $C>0$ depends only on the uniform bounds in the assumptions and the original system state $\hu$.
	\end{lemma}
\begin{proof}
	We recall again that 
		$$
	0 = \ip{\partial_t\hvel^n}{r} - \ip{\hEnthalp^n}{\partial_x r} + \gamma \ip{\abs{\hvel^n}\hvel^n}{r} + \hEnthalp^n_\partial r\biggl\lvert_{x=0}^{x=\ell} - \ip{\src_2^n}{r},
	$$
	for any $r \in H^1(\xDom)$. We chose $r = \numnM^n - \numhM^n$, and add this to $- \ip{\res_2^n}{\numnM^n - \numhM^n} $, getting
	\begin{equation}
	\begin{aligned}
		- \ip{\res_2^n}{\numnM^n - \numhM^n}  =& \ip{\partial_t\hvel^n - \ddt(\numhvelNL{n})}{\numnM^n - \numhM^n} - \ip{\hEnthalp^n - \numhEnthalp^n}{\partial_x (\numnM^n - \numhM^n)} \\
		&+ \gamma \ip{\abs{\hvel^n}\hvel^n - \abs{\numhvelNL{n}}\numhvelNL{n}}{\numnM^n - \numhM^n} + \ip{\src_2^n - \src_2^n}{\numnM^n - \numhM^n}\\
		& + (\hEnthalp_{\partial}^n-\hEnthalp_{\partial}^n)(\numnM^n - \numhM^n)\biggl\lvert_{x=0}^{x=\ell}\eqqcolon (\tilde A) + (\tilde B) + (\tilde C)+ 0 + 0.
	\end{aligned}
	\end{equation}
	The boundary terms vanish by definition of the functions $\numnM,\numhM$.
	Next, we estimate $(\tilde A)$, getting
	\begin{align}
		(\tilde A) = \ip{\partial_t\hvel^n - \ddt(\numhvelNL{n})}{\numnM^n - \numhM^n} \leq \cpoin \ell ||\partial_t\hvel^n - \ddt(\numhvelNL{n})||_2 ||\partial_x(\numnM^n - \numhM^n)||_2.
	\end{align}
	Using Lemma \ref{lemma:dt_proj_estimates} and
	applying Young's inequality, gives
	\begin{align}
	  ||\partial_t\hvel^n - \ddt(\numhvelNL{n})||_2 ||\partial_x(\numnM^n - \numhM^n)||_2 &\leq   {\alpha}||\numdens^n - \numhdens^n||_2^2 + \frac{C^2 }{4\alpha}(\tStep^2 + \xStep^2)
	\end{align}
	for any $\alpha \in \symbb{R}^+$.
	For the term $(\tilde B)$, we use estimate \eqref{eqn:enthalp_proj_est} to bound the projection error $||\hEnthalp^n - \numhEnthalp^n||_2 \leq C \xStep$. Then, for any $\alpha \in \symbb{R}^+$,  by weighted Young's inequality we obtain
	\begin{align}
		\ip{\hEnthalp^n - \numhEnthalp^n}{\partial_x (\numnM^n - \numhM^n)}  \leq \alpha||\numdens^n - \numhdens^n||_2^2 + \frac{C}{\alpha}(\tStep^2 + \xStep^2),
	\end{align}
	 with $C$ depending only on the bounds in \ref{ass:unif}.

	Then, we consider the friction term $(\tilde C)$ and use estimate \eqref{eqn:fric_est}. Together with the uniform bounds \ref{ass:unif} and Lemma \ref{lemma:dt_proj_estimates}, we have 
	$$
	||\,\abs{\hvel^n}\hvel^n - \abs{\numhvelNL{n}}\numhvelNL{n}\,||_2 \leq 2\overbar{\vel} ||\hvel^n -\numhvelNL{n}||_2 \leq C\xStep .
	$$
		For any $\alpha \in \symbb{R}^+$, this gives the estimate 
	\begin{equation}
	\begin{aligned}
		 \gamma \ip{\abs{\hvel^n}\hvel^n - \abs{\numhvelNL{n}}\numhvelNL{n}}{\numnM^n - \numhM^n} &\leq\gamma\cpoin \ell ||\,\abs{\hvel^n}\hvel^n - \abs{\numhvelNL{n}}\numhvelNL{n}\,||_2 ||\partial_x (\numnM^n - \numhM^n)||_2 \\
		 &\leq \gamma\cpoin\ell \frac{C^2\xStep^2 }{\alpha}+  \gamma\cpoin \ell \alpha||\numdens^n - \numhdens^n||_2^2 
	\end{aligned}
\end{equation}

\end{proof}
\begin{remark}[Boundary Conditions]
	\label{rem:boundary_conds}
	The proofs of Lemma \ref{lemma:disG_est_1} and Lemma \ref{lemma:ddt_disG_II} show why we need a boundary condition on $\hEnthalp_\partial$ at one end of the pipe and a boundary condition on $\hflow_\partial$ at the other end. In equation \eqref{eqn:dt_Mdif}, we get the term $\numflow^n(0) -\numhflow^n(0)$. This term must be made to vanish, because we have no other suitable estimates to control it.
	
	Similarly, the boundary term in equation \eqref{eqn:enthalp_boundary} has to vanish, but we cannot guarantee ${\numnM^n(\ell) - \numhM^n(\ell) = 0} $, we require the boundary condition on the enthalpy instead.
\end{remark}
We summarize the estimates of the previous lemmas in the following result: 
\begin{theorem}[Estimate of $\ddt \disG(\numu^n,\numhu^n)$]
		Let $\numhu^n$ be the projected original system state given by \eqref{def:proj_state} and let $\numu^n$ be a solution of system \eqref{eqn:dis_scheme_1} -- \eqref{eqn:dis_scheme_2} and $\delta$ the constant from the modified relative energy \eqref{eqn:mod_rel_energy}. Let $0 < \tStep \leq \tStep_0$ and $0 < \xStep \leq \xStep_0$. Under Assumptions \ref{ass:unif} --\ref{ass:subsonic}, \ref{ass:discrete} and \ref{ass:meas_err}, there exist constants $C_1,C_2,C_3> 0$ such that for any $\alpha \in \symbb{R}^+$, the following estimate holds:
		\begin{equation}
			\begin{aligned}
				\ddt \disG(\numu^n,\numhu^n) \leq& \left(-\frac 12\underline{C}_{\ppot''}+ \frac{\overbar{\vel}^2}{2\overbar{\dens}} +\frac{\overbar{\vel}^2}{2\underbar{\dens}}+ \alpha\right)||\numdens^n -\numhdens^n||_2^2 
				\\&+\left(\frac 32 \overbar{\dens} + \frac 12 \underbar{\dens} +  \frac{\cpoin^2 \ell^2}{\underline{C}_{P''}}(\mu^2 + \gamma^24\overbar{\vel}^2)+\alpha \right)||\numvel^n - \numhvelNL{n}||_2^2\\
				&+ C_1(1 + \frac{1}{\alpha})(\tStep^2 + \xStep^2)+ \frac{\mu^2 C_2}{\alpha}||\error||_{\besov{\infty}{2}}^2 
				\\&+ \frac{\tStep}{2\delta}||\ddt(\numu^n - \numhu^n)||_2^2 + C_3\delta \tStep ||\numu^n - \numhu^n||_2^2
			\end{aligned}
		\end{equation}
		The constants $C_1,C_2$ and $C_3$ are independent of $n$, $\tStep$ and $\xStep$.
\end{theorem}
\section{Uniform Convergence Result}
\label{sec:conv_result}
We are now in a position to prove the final uniform convergence result. It essentially remains to show that for any $\mu \in \symbb{R}^+$, one can choose $\delta$ and the additional free parameter $\alpha$ from the various estimates relying on weighted Young's inequalities, such that there exist constants $C,\tilde C>0$ such that
$$
\ddt \left(\hamil(\numu^n,\numhu^n) + \delta \disG(\numu^n,\numhu^n)\right) \leq -C \left(\hamil(\numu^n,\numhu^n) + \delta \disG(\numu^n,\numhu^n)\right) + \tilde C(\tStep^2 + \xStep^2),
$$ 
and $\tilde C$ is independent of $n$.
Such an estimate only holds true for sufficiently small bounds $C_t$ and $\overbar{\vel}$ from Assumption \ref{ass:smallderivs}.
  In particular, one needs to impose a small enough bound on $\overbar{v}$ for Lemma \ref{lemma:ddt_disG_II} to give any dissipation in $||\numdens^n -\numhdens^n||_2^2$. Further, due to the estimate in Lemma \ref{lemma:dth2}, one needs $C_t$ small enough to not dominate the dissipation one gets.

 This is made more precise in the following lemma. After that, it remains to apply a modified version of a discrete Gronwall's lemma to show the convergence result.
\begin{lemma}
	\label{lemma:negtimederiv}
	Let $\numhu^n$ be the projected original system state given by \eqref{def:proj_state} and $\numu^n$ be a solution of \eqref{eqn:dis_scheme_1} -- \eqref{eqn:dis_scheme_2}. Further,  let Assumptions \ref{ass:unif} -- \ref{ass:smallderivs}, \ref{ass:meas_err} and \ref{ass:discrete} hold. Let $\alpha \in \symbb{R}^+,$ $\delta \in \symbb{R}^+$ and the bounds $\overbar{\vel}, C_t>0$ be sufficiently small.
	
	 Then, for any $\mu \in \symbb{R}^+$, there exist $\tau_0, \xStep_0> 0$, a constant  $C_1 = C_1(\alpha,\tStep_0)>0$ and constants $C_2,C_3,C_4> 0$ such that for all $0<\tau \leq \tau_0$ and $0<\xStep \leq \xStep_0$, we have
	 \begin{equation}
		\begin{aligned}
			\ddt(\hamil(\numu^n,\numhu^n) &+ \delta \disG(\numu^n, \numhu^n))\\ 
			\leq &\, -C_1(\alpha,\tStep_0) ||\numu^n - \numhu^n||_2^2 + (\alpha + \tStep_0)C_{2}||\numu^{n-1} - \numhu^{n-1}||_2^2 - \frac 12 \diss(\numu^n,\numhu^n) \\
			&+(1 + \frac{1}{\alpha})(\tStep^2 +\xStep^2)C_3 + \frac{\mu^2 C_4}{\alpha}||\error||_{\besov{\infty}{2}}^2 + \ddt(\ip{\hEnthalp^n - \numhEnthalp^n}{\numdens^n - \numhdens^n}),
		\end{aligned}	
	 \end{equation}
	 The constants $C_1,C_2,C_3$ and $C_4$ are independent of $n$, $\xStep$ and $ \tStep$.
	Further, $\alpha$ and $\tStep_0$ can be chosen small enough, such that
	\begin{equation}
	-C_1(\alpha,\tStep_0) + (\alpha + \tStep_0){C}_2 < 0.
	\end{equation}  
\end{lemma}
\begin{proof}
	We begin by gathering all the contributions from the previous estimates, which also occur in the estimates for the continuous observer system and have no free parameter able to make them arbitrarily small.
	This gives 
	\begin{equation}
		\begin{aligned}
			\ddt(\hamil(\numu^n,\numhu^n) &+ \delta \disG(\numu^n, \numhu^n))\\ \leq &\left(- \frac{\mu}{2}\underbar\dens+\frac {3\delta}{2}\overbar\dens +\frac {\delta}2 \underbar{\dens}+\frac{\delta\cpoin^2 \ell^2}{\underline{C}_{P''}}(\mu^2 + 4\gamma^2\overbar{\vel}^2) \right)||\numvel^n-\numhvel^n||_2^2 \\
			&+\left(\frac 12 \mu \frac{\overbar{v}^2}{\underbar{\dens}} +2\gamma \frac{\overbar{v}^3}{\underline{\rho}}+ \delta\frac{\overbar{\vel}^2}{2\overbar\dens} + \delta\frac{\overbar{\vel}^2}{2\underbar\dens}
			-\frac{\delta}{2}\underline{C}_{\ppot''}\right) ||\numdens^n - \numhdens^n||_2^2\\
			&+\frac 12 \max(1,\overbar{C}_{\ppot'''}) \left(||\partial_t\hdens^n||_\infty + ||\partial_t\hvel^n||_\infty\right)|| \numu^n - \numhu^n||_2^2 \\
			&+ \text{relative dissipation, measurement error, $\alpha$, and $\symcal{O}(\tStep^2 + \xStep^2)$ terms.}
		\end{aligned}
		\label{eqn:dissipation_rate}
	\end{equation}
	Note that the numerical dissipation term from eqn. \eqref{eqn:numdiss}, $-\frac 12 \tStep \underline{C} ||\ddt(\numu^n -\numhu^n )||_2$, and the additional term from \eqref{eqn:numdiss2} cancel.
Further, for any fixed $\mu >0$, there exists a constant $C_\vel>0$, depending only on the bounds in \ref{ass:unif}, \ref{ass:discrete} and the parameter $\mu$, such that we can always choose $\delta$  and $C_t =||\partial_t\hdens||_{\besov{\infty}{\infty}} + ||\partial_t\hvel||_{\besov{\infty}{\infty}}$ small enough, that 
\begin{equation}
	\label{eqn:vel_diss_const}
	\left(- \frac{\mu}{2}\underbar\dens+\frac {3\delta}{2}\overbar\dens +\frac {\delta}2 \underbar{\dens}+\frac{\delta\cpoin^2 \ell^2}{\underline{C}_{P''}}(\mu^2 + 4\gamma^2\overbar{\vel}^2) \right)  +\frac 12 \max(1,\overbar{C}_{\ppot'''}) C_t \leq -C_\vel< 0.
\end{equation}
For any fixed $\mu$, given this choice of $\delta$, there exists a constant $C_\dens>0$, such that it is possible to choose $\overbar{v}$ and $C_t$ small enough for the following inequality to hold:
\begin{equation}
	\label{eqn:dens_diss_const}
	\left(\frac 12 \mu \frac{\overbar{v}^2}{\underbar{\dens}} +2\gamma \frac{\overbar{v}^3}{\underline{\rho}}+\delta\frac{\overbar{\vel}^2}{2\overbar\dens} + \delta\frac{\overbar{\vel}^2}{2\underbar\dens}
	-\frac{\delta}{2}\underline{C}_{\ppot''}\right)+\frac 12 \max(1,\overbar{C}_{\ppot'''}) C_t \leq -C_\dens< 0.
\end{equation} The constants $C_\vel$ and $C_\dens$ depend on the uniform bounds in the assumptions \ref{ass:unif} -- \ref{ass:smallderivs} and \ref{ass:discrete}, however, they are independent of $n$ and the discretization parameters $\tStep$ and  $\xStep$.

Gathering all the remaining terms depending on $||\numu - \numhu||_2^2$ at either $t = t^n$ or $t = t^{n-1}$, we get
\begin{align}
(	\text{$\alpha$, $\tStep$ terms} )\leq  (\alpha +  \tStep) \tilde{C}_1 ||\numhu^{n}-\numu^{n}||_2^2 + (\alpha + \tStep)\tilde{C}_2||\numhu^{n-1}-\numu^{n-1}||_2^2,
\end{align}
for some $\tilde{C}_1,\tilde{C}_2>0$ which are independent of $n$.
Finally, collecting all the remaining terms, we have the estimate 
\begin{equation}
	\begin{aligned}
		\ddt(\hamil(\numu^n,\numhu^n) &+ \delta \disG(\numu^n, \numhu^n))\\ 
		\leq &\, -C_1(\alpha,\tStep_0) ||\numu^n - \numhu^n||_2^2 + (\alpha + \tStep_0)\tilde{C}_2||\numu^{n-1} - \numhu^{n-1}||_2^2 - \frac 12 \diss(\numu^n,\numhu^n) \\
		&+C_3(1 + \frac{1}{\alpha})(\tStep^2 +\xStep^2) + \frac{\mu^2 C_4}{\alpha}||\error||_{\besov{\infty}{2}}^2  + \ddt(\ip{\hEnthalp^n - \numhEnthalp^n}{\numdens^n - \numhdens^n})
	\end{aligned}
\end{equation}
 with 
$$
C_1(\alpha,\tStep_0) = \min(C_\vel,C_\dens) -(\alpha + \tStep_0)\tilde{C}_1.
$$
Thus, it is possible to choose $\alpha$ small enough, such that for $\tStep \leq \tStep_0$ sufficiently small, we get $C_1(\alpha,\tStep_0)> 0$  and $-C_1(\alpha,\tStep_0) + (\alpha + \tStep_0)\tilde{C}_2 < 0$.

\end{proof}
\begin{remark}
	\label{rem:conv_rate}
	The previous lemma gives some insight on how to choose $\mu$. When setting the parameter $\mu$, there are three main goals: 
	\begin{itemize}
		\item obtaining a fast exponential decay rate, i.e. having $C_1(\alpha,\tStep_0)$ as large as possible, 
		\item having a wide range of applicability, that is, being able to prescribe large values of $\overbar{v}$ and $C_t$ from Assumption \ref{ass:smallderivs}, and
		\item balancing the amplification of measurement errors.
	\end{itemize}
	Contrary to the intuition one might have, that larger $\mu$ will always provide faster convergence, because the velocity component is nudged more strongly towards to the data, the above proofs indicate that overly large $\mu$ will be suboptimal. In particular, a larger choice of $\mu$ requires a smaller choice of $\delta$ in the estimate \eqref{eqn:vel_diss_const}. A smaller choice of $\delta$ then also needs more restrictive bounds on $\overbar{v}$ and $C_t$ to show estimate \eqref{eqn:dens_diss_const}.
	
	Further, choosing larger values of $\mu$ also amplifies measurement errors more strongly, reducing the maximum possible accuracy of the observer. This is in line with the motivation of using an observer in the first place. In the formal limit of $\mu \to \infty$, the observer system corresponds to setting $\vel =\hvel$ and reconstructing $\hdens$ using the mass conservation equation of the Euler system. This means  the relevant measurement error is the derivative $\partial_x\error$. Typically, measurement errors for the derivatives of some quantity are much larger than those of that quantity itself. In our analysis, we have no regularity assumptions on the derivatives of the measurement error, and thus, our estimate blows up as $\mu \to \infty$.

\end{remark}
Next, we move on to the final convergence result.
\begin{theorem}[Uniform Convergence]
	\label{thm:unif_conv} 
	Under the assumptions of Lemma \ref{lemma:negtimederiv}, let the constants $\alpha$ and $\tStep_0$ be  chosen small enough, such that $C_1C_0^{-1} >(\alpha + \tStep_0) C_{2}c_0^{-1}$. 
	Then, for all $0 < \tStep \leq \tStep_0$ and $0 <\xStep \leq \xStep_0$  with $\tStep_0$ and $\xStep_0$ sufficiently small, it holds that 
	\begin{equation}
		\begin{aligned}
			\frac 12 c_0 ||\numu^n -\numhu^n||_2^2 + c\tStep&\sum_{k=1}^n \left(\prod_{j=k}^{n}{w}_j^{-1}\right) ||\numvel^k -\numhvelNL{k}||_3^3 \\&\leq 2 C_0 \left(\prod_{j=1}^n {w}_j^{-1}\right) ||\numu^0 -\numhu^0||_2^2 + C(\tStep^2 + \xStep^2) + \mu^2 C||\error||_{\besov{\infty}{2}}^2 .
		\end{aligned}
	\end{equation}

	The weights $ w_n$ are defined by
	$$
	{w}_n \coloneqq \frac{2 + \tau C_1 C_0^{-1}}{2 + \tau(\alpha + \tStep_0) C_2 c_0^{-1}},
	$$
with the respective $C_1$ and $C_2$ from Lemma \ref{lemma:negtimederiv}
and
 the constants $c_0, C_0$ and $C$ are independent of $n$, $\tStep$ and $\xStep$.
\end{theorem}
\begin{remark}
	Expanding the weights $ w_n$ in $\tStep$ around $\tStep = 0$ gives $$ w_n = 1 + \frac 12\tStep (C_1C_0^{-1} - (\alpha + \tStep_0)C_2 c_0^{-1}) + \symcal{O}(\tStep^2).$$ Thus, the product $\prod_{j=1}^n{w}_j^{-1}$ is a discrete approximation of $$\exp(-\frac{1}{2}(C_1C_0^{-1} - (\alpha + \tStep_0)C_{2}c_0^{-1})t^n).$$ Recalling the statement of Lemma \ref{lemma:negtimederiv}, one can choose $-C_1C_0^{-1} + (\alpha + \tStep_0)C_{2}c_0^{-1}< 0$ to be arbitrarily close to the continuous decay rate of the modified relative energy as proven in \cite{kunkel_obs}, at the cost of making the constant on the $\symcal{O}(\tStep^2 + \xStep^2)$ term worse. 
	
	Thus, we show exponential convergence with respect to $ ||\numu^0-\numhu^0||_2^2$, analogously to the continuous case, however, with a constant in the exponent that is (about) half as large.
	In the proof of Theorem \ref{thm:unif_conv}, we see that for any $0< \beta< 1$ one can even recover the continuous exponent up to a factor of $(1 - \beta)$, again, at the cost of making the constants of the discretization errors worse.
\end{remark}
\begin{proof}[Proof of Theorem \ref{thm:unif_conv}]
	For brevity, in the following, we set 
	$$
	a_n \coloneqq \hamil(\numu^n,\numhu^n) + \delta \disG(\numu^n,\numhu^n).
	$$
	Further, we set 
	\begin{equation*}
		\begin{aligned}
		&	g_n \coloneqq - \frac 12 \diss(\numu^n,\numhu^n)+ C_3(1 + \frac{1}{\alpha})(\tStep^2 +\xStep^2) +  \frac{\mu^2 C_4}{\alpha}||\error||_{\besov{\infty}{2}}^2\quad \text{and}\\ 
		&f_n \coloneqq \ip{\hEnthalp^n - \numhEnthalp^n}{\numdens^n - \numhdens^n}.
		\end{aligned}
	\end{equation*}
	Additionally, we omit the dependencies of the constants from the previous lemma on the parameters $\alpha$ and $\tStep_0$. We set $\tilde C_2 := (\alpha + \tStep_0) C_2$.
	Using the norm equivalence of the modified relative energy \eqref{eqn:mod_rel_energy}, we have 
	\begin{equation}
		\begin{aligned}
			\ddt a_n  
			\leq \,& -C_1C_0^{-1} a_n + \tilde C_{2}c_0^{-1}a_{n-1} + g_n + \ddt f_n.
		\end{aligned}
	\end{equation}
	At this point, we wish to apply a suitable variant of Gronwall's inequality. Because the term $\ddt(\ip{\hEnthalp^n - \numhEnthalp^n}{\numdens^n - \numhdens^n})$ needs special treatment, we give a proof here.
	
	Let $\tilde{a}_n \coloneqq a_n\prod_{j=1}^n {w}_j$, with weights ${w}_j > 1$, depending on the time step $\tStep$, which are still to be specified in the following steps of the proof. We now compute the discrete time derivative: 
	\begin{equation}
		\begin{aligned}
			\ddt \tilde{a}_n &^= \frac{{w}_n-1}{\tStep}\left(\prod_{j=1}^{n-1}{w}_j\right) a_n + \left(\prod_{j=1}^{n-1}{w}_j\right)  \ddt a_n \\
			&\leq \frac{{w}_n-1}{\tStep}\left(\prod_{j=1}^{n-1}{w}_j\right) a_n + \left(\prod_{j=1}^{n-1}{w}_j\right) \left(-C_1C_0^{-1} a_n +\tilde C_{2}c_0^{-1}a_{n-1}+ g_n + \ddt f_n\right).
		\end{aligned}
	\end{equation}
	Recall that choosing $\alpha$  and $\tStep_0$ small enough, one also has ${-C_1C_0^{-1} +\tilde C_{2}c_0^{-1} < 0}$.		
	Next, we apply discrete integration by parts to the $\ddt f_n$ term: 
	\begin{equation}
		\left(\prod_{j=1}^{n-1}{w}_j\right) \ddt f_n = - \frac{{w}_n-1}{\tStep}\left(\prod_{j=1}^{n-1}{w}_j\right) f_n + \ddt\left(f_n\prod_{j=1}^{n}{w}_j \right).
	\end{equation}
	For $f_n$, by applying Hölder's and weighted Young's inequality with $\beta \in \symbb{R}^+$, we have 
	\begin{equation}
		|f_n| \leq \frac{1}{4\beta}||\hEnthalp^n - \numhEnthalp^n||_2^2 + \beta||\numdens^n - \numhdens^n||_2^2 \leq \frac{C}{\beta} \xStep^2 + \beta c_0^{-1} a_n.
		\label{eqn:ddtfn}
	\end{equation}
	We set $\beta = c_0$ as the constant from the norm equivalence. Thus, overall, we have 
	\begin{align}
		\ddt \tilde{a}_n &\leq 2\frac{{w}_n-1}{\tStep}\left(\prod_{j=1}^{n-1}{w}_j\right) a_n + \left(\prod_{j=1}^{n-1}{w}_j\right) \left(-C_1C_0^{-1} a_n + \tilde C_{2}c_0^{-1}a_{n-1}+ \tilde g_n\right)+\ddt\left(f_n\prod_{j=1}^{n}{w}_j \right),
		\label{eqn:tderiv_atil}
	\end{align}
	where $\tilde g_n \coloneqq g_n + C c_0^{-1}\xStep^2$. We now choose $ w_n$ such that 
	\begin{equation}
		2\frac{{w}_n-1}{\tStep} - C_1C_0^{-1} = -\tilde C_2 c_0^{-1} w_n,
	\end{equation}
	or in other words, ${w}_n \coloneqq\frac{2 + \tau C_1 C_0^{-1}}{2 + \tau \tilde C_2 c_0^{-1}}$. Due to ${C_1C_0^{-1} > \tilde C_{2}c_0^{-1} }$, we have $ w_n>1$, such that this choice is compatible with our previous assumptions.
	
	Inserting this choice of $ w_n$ into \eqref{eqn:tderiv_atil} gives
	\begin{align}
		\ddt \tilde a_n \leq \tilde C_2 c_0^{-1}\left(-\tilde a_n + \tilde a_{n-1}\right) + \left(\prod_{j=1}^{n-1}{w}_j\right)\tilde g_n + \ddt\left(f_n\prod_{j=1}^{n}{w}_j \right),
	\end{align}
	and rearranging leads to 
	\begin{align}
	(1 +  \tilde C_2 c_0^{-1}\tStep)\ddt \tilde a_n \leq  \left(\prod_{j=1}^{n-1}{w}_j\right)\tilde g_n + \ddt\left(f_n\prod_{j=1}^{n}{w}_j \right).
	\end{align}
	Next, multiplying by $\tStep$ and summing up gives
	\begin{align}
		\tilde a_n \leq \tilde a_0 + \frac{1}{1 + \tilde C_{2}c_0^{-1}\tau} \left(\tStep\sum_{k=1}^n \left( \prod_{j=1}^{k-1}{w}_j\right) \tilde{g}_k + f_n\left(\prod_{j=1}^{n}{w}_j\right) - f_0\right).
	\end{align}	
	Finally, dividing both sides by $\prod_{j=1}^n {w}_j$ and the fact that $\tilde a_0 = a_0$ gives 
	\begin{align}
		a_n \leq \left(\prod_{j=1}^n {w}_j^{-1}\right) a_0 + \frac{1}{1 + \tilde C_{2}c_0^{-1}\tau} \left(\tStep\sum_{k=1}^n \prod_{j=k}^{n}{w}_j^{-1} \tilde{g}_k  + f_n - \left(\prod_{j=1}^n {w}_j^{-1}\right)f_0\right).
	\end{align}
	Lastly, estimating $f_n$ and $f_0$ again as in eq. \eqref{eqn:ddtfn}, for suitable choices of $\beta$, we arrive at 
	\begin{align}
	\frac 12	a_n \leq 2 \left(\prod_{j=1}^n {w}_j^{-1}\right) a_0 + \frac{\tStep}{1 + \tilde C_{2}c_0^{-1}\tau} \sum_{k=1}^n \prod_{j=k}^{n}{w}_j^{-1} \tilde{g}_k + C \xStep^2.
	\end{align}
	In particular, $C$ is independent of $n$, $\tStep$ and $\xStep$. Then inserting the definitions of $\tilde{g}_n$, the fact that $\prod_{j=k}^{n}{w}_j^{-1} \leq 1$  and using $\diss(\numu^k,\numhu^k) \geq c ||\numvel^k -\numhvelNL{k}||_3^3$, we obtain
		\begin{align}
		\frac 12 a_n + c\tStep\sum_{k=1}^n \left(\prod_{j=k}^{n}{w}_j^{-1}\right) ||\numvel^k -\numhvelNL{k}||_3^3 \leq 2 \left( \prod_{j=1}^n {w}_j^{-1}\right) a_0 + C(\tStep^2 + \xStep^2) + C||\error||_{\besov{\infty}{2}}^2.
	\end{align}
	Then, using $a_n \geq c_0 ||\numu^n -\numhu^n||_2^2$ and $a_0 \leq C_0 ||\numu^0 -\numhu^0||_2^2$ completes the proof.
\end{proof}
Finally, using the projection error estimates for the original system state, we have the following corollary:
\begin{corollary}
	\label{col:unif_conv}
	Let the assumptions of Theorem \ref{thm:unif_conv} hold. Let $\hu$ be a solution of system \eqref{eqn:masscons} -- \eqref{eqn:momcons} and $\numu$ be a solution of \eqref{eqn:dis_scheme_1} -- \eqref{eqn:dis_scheme_2}. Then 
	\begin{equation}
		\begin{aligned}
			\frac 12 c_0 ||\numu^n -\hu^n||_2^2 &\leq 2 C_0 \left(\prod_{j=1}^n {w}_j^{-1}\right) ||\numu^0 -\numhu^0||_2^2 + \tilde C(\tStep^2 + \xStep^2) + \mu^2 C||\error||_{\besov{\infty}{2}}^2,
		\end{aligned}
	\end{equation}
	with the constants as introduced in Theorem \ref{thm:unif_conv} and $\tilde C$ a modified constant, that is still independent of $n$, $\tau$ and $\xStep$.
\end{corollary}
\begin{proof}
	The corollary is a direct consequence of Theorem \ref{thm:unif_conv} and projection error estimates for the original system state $\hu$, given the regularity from Assumption \ref{ass:unif} and in particular that $\hu$ is assumed to be uniformly Lipschitz. More precisely, using Lemma \ref{lemma:proj_estimates}, we have 
	\begin{equation}
		|| \hdens^n  - \numhdens^n||_2 \leq c \xStep ||\partial_x\hdens||_{\besov{\infty}{2}} \leq C \xStep
	\end{equation}
	 with $C > 0$ independent of $n$. Further, from Lemma \ref{lemma:dt_proj_estimates}, we have 
	 \begin{equation}
	 	||\numhvelNL{n} - \hvel^n||_2 \leq C \xStep.
	 \end{equation}
	 Then, the error splitting
	 \begin{equation}
	 	||\numu^n -\hu^n||_2 \leq ||\numu^n -\numhu^n||_2  + ||\numhu^n -\hu^n||_2 
	 \end{equation}
	 shows the statement of the corollary.
\end{proof}

%% file: Sections/numerics.tex
\section{Numerical Experiments}
\label{numerics}
In this section we present several numerical experiments to illustrate the error bound. In particular, we want to highlight the following three points: 
\begin{itemize}
	\item We observe an initial exponential decay of the error of the numerical solution, followed by a plateau in error, which does not increase in time.
	\item We observe that the height of the plateau is proportional to $\xStep$ and $\tStep$.
	\item How quickly the error approaches the plateau depends on the nudging parameter $\mu$, with overly large $\mu$ leading to slower synchronization, see Remark \ref{rem:conv_rate}.
\end{itemize}
To compare the numerical solution of the observer system to the original system state, we cannot solve the original system using some numerical scheme. Unlike the observer system, the reference solution computed this way could develop large errors over time, making it unsuitable for comparison.
 Because of this, we instead use manufactured solutions, that is, we prescribe a function $\hu$ and then compute source terms $\src_1$ and $\src_2$ such that equations \eqref{eqn:masscons} -- \eqref{eqn:momcons} are satisfied.
\subsection{Implementation details}
A reference implementation of the numerical scheme for the observer and the required data to generate the plots is located at {\url{https://git-ce.rwth-aachen.de/publications_chaumet/observers}}.

For the source term $\src_1$, the numerical algorithm requires evaluating the $L^2$ projection $\Pi_\xStep$, i.e. it is required to compute cell-wise averages of $\src_1$. We do this using the trapezoid rule. Consider a cell $[x_i,x_{i+1}]$ with $0 \leq i \leq M-1$, then we have $\Pi_\xStep(\src_1^n)\big\lvert_{[x_i,x_{i+1}]} = \frac 12 ( \src_1^n(x_i) + \src_1^n(x_{i+1})) + \symcal{O}(\xStep^2)$. Handling the $\symcal{O}(\xStep^2)$ terms similarly to the measurement error, one can see that these errors will not accumulate, and because they are at second order in $\xStep$, they will not significantly influence the plateau of the error. Computing the $L^2$ errors $||\numu^n - \hu^n||_2$ requires numerical integration and we use the trapezoid rule for this as well.

We solve the nonlinear system of equations \eqref{eqn:dis_scheme_1} -- \eqref{eqn:dis_scheme_2} using the Newton method, taking the discrete solution of the previous time step as an initial guess in each new time step. We prescribe a solver tolerance of at least $10^{-12}$ on the residuals with respect to the maximum norm, however, in the following experiments, we always achieve residuals less than $10^{-15}$ for each time step.
\subsection{Exponential decay and convergence order}
For the following numerical examples we consider the ideal gas law, $p(\hdens) = \hdens$. The corresponding pressure potential is $\ppot(\hdens) = \hdens\log(\hdens)$. We set $\xDom = [0,1]$ and run our simulations until $t^n = 40$ for some $n \in \symbb{N}$.
 We then choose the following manufactured solution: 
\begin{align}
	\hdens(x,t) &= 2 + \sin(\pi(x+t)), \quad \text{and} \\
	\hvel(x,t) &= \frac 1{10} \sin(\pi x) \cos(\pi t).
\end{align}
 We set the friction parameter $\gamma = 0.1$ and choose $\mu = 1$ for the first experiment.
Then, $(\hdens,\hvel)$ is a solution of system \eqref{eqn:obs_masscons} -- \eqref{eqn:obs_momcons} with the boundary conditions $\hflow_\partial\lvert_{x=0} \coloneqq 0$ and $\hEnthalp_\partial\lvert_{x=1}(t) \coloneqq 1 + \log(2 + \sin(\pi(1 + t)))$ with 
\begin{align}
	\src_1(x,t) \coloneqq& \pi\cos(\pi (t + x)) + \frac {1} {10}\pi\cos (\pi t) (\sin (\pi (t + 2 x)) + 
	2\cos (\pi x)), \quad \text{and}\\
	\src_2(x,t) \coloneqq& \frac {1} {100}\sin (\pi x)\left (\gamma\cos (\pi t) \abs{ \cos (\pi \
	t)\sin (\pi x)}  +\pi\cos^2 (\pi t)\cos (\pi x) - 
	10\pi\sin (\pi t) \right) \\\nonumber&+ \frac {\pi\cos (\pi (t + 
		x))} {\sin (\pi (t + x)) + 2}.
\end{align}
\begin{figure}[h]
	\includegraphics[width = 0.5\textwidth]{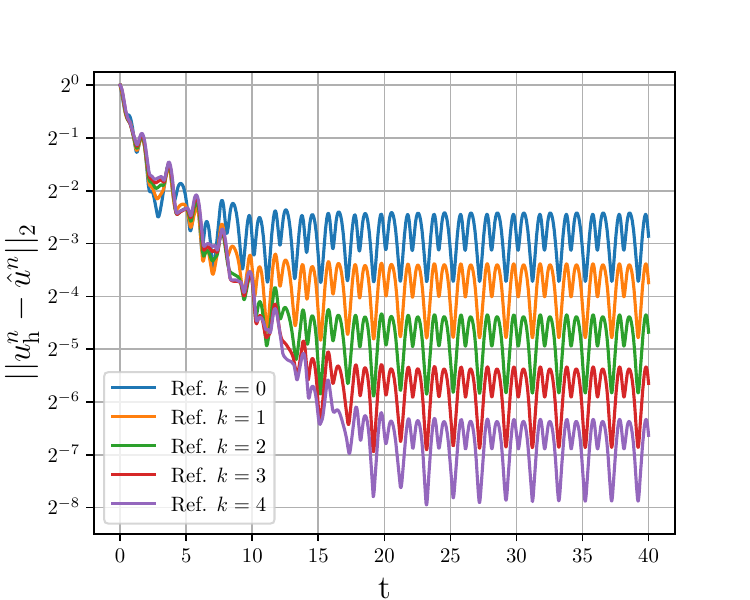}
	\caption{Absolute $L^2$ error over time for $\mu = 1$. The error is normalized such that $||\numu^0 - \hu^0||_2 = 1$ at $t=0$.}
	\label{fig:err_mu1}
\end{figure}
For the following numerical studies, we set the following refinement levels for the spatial- and temporal grids: $M_k \coloneqq 30\cdot 2^k$ and $N_k = 1200\cdot 2^k$, such that $\xStep_k = \tStep_k$. We assume that the density $\hdens$ is unknown and initialize the observer with $\numdens^0(x) = 2.5\,$.

 Note that, because we do not explicitly verify the smallness assumption \ref{ass:smallderivs}, it is important to initialize the observer with some density that is not too far from the original system's density. This avoids the numerical solution of the observer developing a shock before the observer term has had sufficient time to stabilize the observer's dynamics towards the original system. When the observer system develops a shock, the Newton method does not converge in our experiments anymore.

 Figure \ref{fig:err_mu1} shows the evolution of the $L^2$ error over time for refinement levels $k = 0, \dotsc, 4$. We normalize the error at $t=0$ such that $||\numu^0 - \hu^0||_2 = 1$. We observe that the error initially decays exponentially until it plateaus at some error level that depends on $\xStep$ and $\tStep$. We see that with each refinement by a factor of two, the final plateau of the error also decreases by about a factor of two,  which shows the predicted linear convergence. The curves seem to coincide across all refinement levels during the exponential part of the error decay. This illustrates that the exponential decay rate that is achieved is (approximately) independent of the discretization parameters.
 
 \subsection{Influence of the nudging parameter}
 In this set of simulations, we study the influence of $\mu$ on the exponential decay of the error. We use the same manufactured solution as we used for the first test and the same refinement levels. However, we choose $\mu = 5$ and $\mu = 25$ for the following tests.

\begin{figure}[h]
	\begin{subfigure}{0.45\textwidth}
		\includegraphics[width=\textwidth]{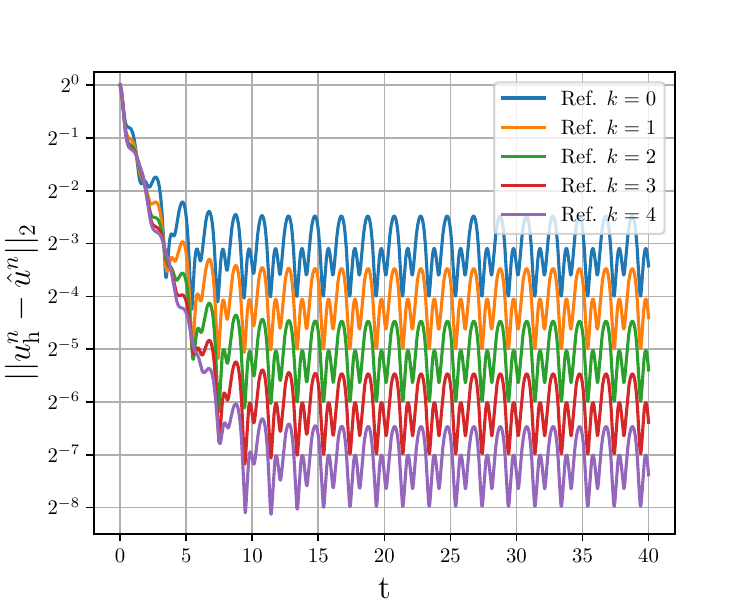}
		\caption{Evolution of the error for $\mu = 5$}
		\label{fig:err_mu5}
	\end{subfigure}
	\begin{subfigure}{0.45\textwidth}
		\includegraphics[width=\textwidth]{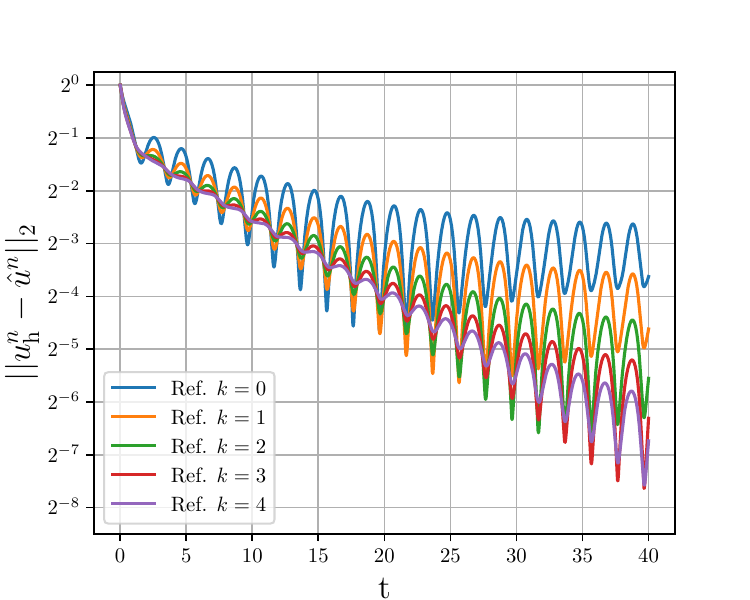}
		\caption{Evolution of the error for $\mu = 25$}
		\label{fig:err_mu25}
	\end{subfigure}
	\caption{Comparison of the convergence speed for different values of $\mu$. Overly large $\mu$ can lead to slow convergence.}
	\label{fig:err_comparison}
\end{figure}
We show the comparison of the convergence speed of the observer in Figure \ref{fig:err_comparison}. In Fig. \ref{fig:err_mu5}, we see that increasing $\mu$ to $5$ has reduced the time for the observer's error to plateau by about a factor of $2$. The finest refinement for $\mu = 5$ reaches its steady-state at about $t = 10$ while the observer with $\mu = 1$ takes up until about $t = 20$ to reach its steady state. In this example, we observe that the final error level that is achieved seems to be mostly unaffected by increasing $\mu$. 

 On the other hand, for $\mu = 25$, in Fig. \ref{fig:err_mu25} we see that the observer converges towards the original system state significantly slower. Its error even fails to plateau for the two finest refinements, within the time interval we consider. However, in the scenarios where the observer does seem to reach the plateau in its error, again, the error level seems quite similar to the final error level for the other two choices of $\mu$.
 This illustrates the importance of choosing a suitable value of $\mu$, even when measurement errors are small or absent entirely.

%% file: Sections/conclusion.tex
\section{Discussion}
In this paper we have proved an error bound for a fully discrete observer for the barotropic Euler equations. We show that the observer term stabilizes the system in such a way that one has an error bound that is uniform-in-time. Thus, the discrete observer stays accurate even for long-time simulations. In particular, we show that the observer converges exponentially with respect to the error in the initial guess of the state, up to additive contributions due to the discretization and measurement errors, which do not increase with time.

Our proof is based on a relative energy approach, which limits it to solutions that are uniformly Lipschitz in time and we require some smallness assumptions to show decay for arbitrary $\mu$. However, the intermediate proof steps also give some insight and criteria on how to choose a “good” value for $\mu$, showing that one has to balance the increase of measurement errors, the applicability of the estimate and the convergence speed. As an extension, one might also consider a time-dependent nudging parameter, that decreases over time, to benefit from a fast exponential decay rate initially, but obtain a more accurate prediction once measurement errors become a significant source of error.

We illustrate our error estimate in a numerical study. While our estimates are certainly not sharp in general and only present an upper bound, the qualitative features of our error estimate are visible in the numerical study, regardless.

To extend the range of applications for observer-based data assimilation for hyperbolic balance laws, it would be of interest to further study how to extend our results to discontinuous solutions.
In similar spirit, in this work, we have only focused on the barotropic Euler equations in one space dimension. The next step is to consider either two dimensions, or an extension to networks. For the network case, \cite{kunkel_obs} shows that the results can be extended to star-shaped networks, and we think the same should be possible for the discrete observer. However, an extension to more general network topologies seems challenging and would require additional understanding at the continuous level first.